\documentclass[11pt]{amsart}

\usepackage[margin=0.88in]{geometry}
\usepackage{amsmath}
\usepackage{amssymb}
\usepackage{algpseudocode}
\usepackage{enumitem}
\usepackage{mdwlist}

\DeclareMathOperator{\supp}{supp}
\DeclareMathOperator{\inid}{\text{in}}
\DeclareMathOperator{\Lex}{Lex\,}
\DeclareMathOperator{\Revlex}{Revlex\,}
\newcommand{\fieldk}{\Bbbk}
\DeclareMathOperator{\Spank}{\text{Span}_{\fieldk}}

\newenvironment{enum}
{\begin{enumerate}

}
{\end{enumerate}}

\newenvironment{enum*}
{\parskip=-4pt
\begin{enumerate*}

}
{\end{enumerate*}}

\theoremstyle{plain}
\newtheorem{theorem}{Theorem}[section]
\newtheorem{proposition}[theorem]{Proposition}
\newtheorem{lemma}[theorem]{Lemma}
\newtheorem{corollary}[theorem]{Corollary}

\theoremstyle{definition}
\newtheorem{definition}[theorem]{Definition}
\newtheorem{remark}[theorem]{Remark}

\begin{document}
\title[Hilbert functions of colored quotient rings]{Hilbert functions of colored quotient rings and a generalization of the Clements-Lindstr\"{o}m theorem}
\author{Kai Fong Ernest Chong}
\address{Department of Mathematics\\
      Cornell University\\
      Ithaca, NY 14853-4201, USA}
\email{kc343@cornell.edu}


\begin{abstract}
Given a polynomial ring $S = \Bbbk[x_1, \dots, x_n]$ over a field $\Bbbk$, and a monomial ideal $M$ of $S$, we say the quotient ring $R = S/M$ is Macaulay-Lex if for every graded ideal of $R$, there exists a lexicographic ideal of $R$ with the same Hilbert function. In this paper, we introduce a class of quotient rings with combinatorial significance, which we call colored quotient rings. This class of rings include Clements-Lindstr\"{o}m rings and colored squarefree rings as special cases that are known to be Macaulay-Lex. We construct two new classes of Macaulay-Lex rings, characterize all colored quotient rings that are Macaulay-Lex, and give a simultaneous generalization of both the Clements-Lindstr\"{o}m theorem and the Frankl-F\"{u}redi-Kalai theorem. We also show that the $f$-vectors of $(a_1, \dots, a_n)$-colored simplicial complexes or multicomplexes are never characterized by ``reverse-lexicographic'' complexes or multicomplexes when $n>1$ and $(a_1, \dots, a_n) \neq (1, \dots, 1)$.
\end{abstract}

\maketitle

\section{Introduction and overview}\label{sec:Intro}
The study of Hilbert functions is one of the central themes in commutative algebra, and much of our understanding is enhanced by insights in combinatorics. This rich interplay between commutative algebra and combinatorics can be traced back to Macaulay's 1927 paper~\cite{Macaulay1927}, in which Macaulay characterized the possible Hilbert functions of graded ideals in the polynomial ring $S:= \fieldk[x_1, \dots, x_n]$.

Macaulay's key idea was that for every graded ideal of $S$, there exists a lexicographic (abbreviated: lex) ideal of $S$ with the same Hilbert function. Lex ideals, sometimes also known as lex-segment ideals, are monomial ideals defined combinatorially: Let $<_{\ell ex}$ denote the degree-lexicographic order on the monomials in $S$ induced by the linear order $x_1 > \dots > x_n$. A {\it lex ideal} of $S$ is a monomial ideal $L$ such that if $m, m^{\prime}$ are monomials in $S$ satisfying $m\in L$, $\deg(m) = \deg(m^{\prime})$ and $m <_{\ell ex} m^{\prime}$, then $m^{\prime} \in L$. These lex ideals play a crucial role in Hartshorne's proof~\cite{Hartshorne1966:ConnectednessHilbertScheme} that Grothendieck's Hilbert scheme is connected, and every lex ideal of $S$ has maximal Betti numbers among all graded ideals with the same Hilbert function~\cite{Bigatti1993:UpperBoundsBetti,Hulett1993:MaxBettiNumbers,Pardue1996:MaxBettiNumbers}. Moreover, in combinatorics, Macaulay's theorem yields numerical characterizations of both the $f$-vectors of multicomplexes, and the $h$-vectors of Cohen-Macaulay complexes~\cite{Stanley1977:CohenMacaulayComplexes}.

Given a monomial ideal $M$ of $S$, we can similarly define lex ideals in the quotient ring $S/M$. Motivated by Macaulay's theorem, Mermin and Peeva~\cite{MerminPeeva2006:LexifyingIdeals} asked if there is an analogous characterization of the Hilbert functions of graded ideals of $S/M$, thereby introducing the notion of Macaulay-Lex rings. In this general context, we say $S/M$ is a {\it Macaulay-Lex ring} (or $M$ is a {\it Macaulay-Lex ideal}) if for every graded ideal of $S/M$, there exists a lex ideal of $S/M$ with the same Hilbert function. In 1969, Clements and Lindstr\"{o}m~\cite{ClementsLindstrom1969} generalized Macaulay's theorem by showing that if $2 \leq e_1 \leq \dots \leq e_n \leq \infty$, then $Q := S/\langle x_1^{e_1}, \dots, x_n^{e_n} \rangle$ is Macaulay-Lex (where $x_i^{\infty} = 0$). Such quotient rings $Q$ are called Clements-Lindstr\"{o}m rings, and the case $e_1 = \dots = e_n = 2$, first proven independently by Sch\"{u}tzenberger~\cite{Schutzenberger1959}, Kruskal~\cite{Kruskal1963} and Katona~\cite{Katona1968} around the 1960s, is of particular interest in combinatorics (see, e.g., \cite[Sec. 8]{Greene-Kleitman}), since it yields a numerical characterization of the $f$-vectors of simplicial complexes.

Although basic properties of Macaulay-Lex rings are now well understood~\cite{Mermin2006:LexlikeSequences,MerminPeeva2006:LexifyingIdeals,MerminPeeva2007:HilbertFunctionsLexIdeals,Shakin2001:GeneralizationMacaulayThm,Shakin2005:MonomialIdeals}, a complete list of all Macaulay-Lex ideals is known only for $n\leq 2$, with the case $n=2$ already being quite complicated~\cite{Shakin2001:GeneralizationMacaulayThm}. As for $n\geq 3$, several partial results are known~\cite{Abedelfatah2013:MacLexRings,Shakin2001:GeneralizationMacaulayThm,Shakin2003:PiecewiseLexsegment,Shakin2005:MonomialIdeals}, but otherwise, finding an explicit characterization of all possible Macaulay-Lex ideals for arbitrary $n$ remains a wide open problem. Recently, Mermin and Murai~\cite{MerminMurai2010} constructed a new class of Macaulay-Lex rings that they call colored squarefree rings. Their construction was inspired by the Frankl-F\"{u}redi-Kalai theorem~\cite{FFK1988}, a combinatorial result that gives a numerical characterization of the $f$-vectors of colored complexes, which they refined further; see \cite[Remark 2.12]{MerminMurai2010} for details.

In this paper, we extend colored squarefree rings and Clements-Lindstr\"{o}m rings to a common class of quotient rings. The combinatorial analogues of these two classes of rings are colored complexes and ``uncolored'' multicomplexes respectively, both of which are special cases of generalized colored multicomplexes. Motivated by this observation, we define colored quotient rings:

\begin{definition}\label{defn-coloredQuot}
Let ${\bf a} = (a_1, \dots, a_n)$, $\boldsymbol{\lambda} = (\lambda_1, \dots, \lambda_n)$ be $n$-tuples such that $1\leq a_i \leq \infty$, $1\leq \lambda_i < \infty$ for each $i$, and let $X_{\boldsymbol{\lambda}} := \bigsqcup_{i=1}^n X_i$ be a set of variables, where $X_i = \{x_{i,1}, \dots, x_{i,\lambda_i}\}$ for each $i$. Fix a linear order on $X_{\boldsymbol{\lambda}}$ by $x_{i,j} > x_{i^{\prime}, j^{\prime}}$ if $j > j^{\prime}$; or $j = j^{\prime}$ and $i < i^{\prime}$. Let $\fieldk[X_{\boldsymbol{\lambda}}]$ be a polynomial ring on the set of variables $X_{\boldsymbol{\lambda}}$ over a field $\fieldk$, graded by $\deg(x_{i,j}) = 1$, and let $Q_{\bf a} := \sum_{i=1}^n \langle X_i \rangle^{a_i+1}$ be a monomial ideal of $\fieldk[X_{\boldsymbol{\lambda}}]$, where $\langle X_i\rangle^{\infty} = 0$. A {\it colored quotient ring} of {\it type} ${\bf a}$ and {\it composition} $\boldsymbol{\lambda}$ is the quotient ring $W := \fieldk[X_{\boldsymbol{\lambda}}]/Q_{\bf a}$.
\end{definition}

Colored quotient rings include both colored squarefree rings and Clements-Lindstr\"{o}m rings. Specifically, a colored squarefree ring is a colored quotient ring of type $(1, \dots, 1)$ and composition $(\lambda_1, \dots, \lambda_n)$ satisfying $\lambda_1 \geq \dots \geq \lambda_n$, while a Clements-Lindstr\"{o}m ring is a colored quotient ring of composition $(1, \dots, 1)$ and type $(a_1, \dots, a_n)$ satisfying $a_1 \leq \dots \leq a_n$. Our first main result is the following characterization of all possible Macaulay-Lex colored quotient rings.

\begin{theorem}\label{main-thm-MacLex}
A colored quotient ring of type ${\bf a} = (a_1, \dots, a_n)$ and composition $\boldsymbol{\lambda} = (\lambda_1, \dots, \lambda_n)$ is Macaulay-Lex if and only if (at least) one of the following conditions hold:
\begin{enum*}
\item ${\bf a} = (1, \dots, 1, a_{r+1}, \dots, a_n)$, $\boldsymbol{\lambda} = (\lambda_1, \dots, \lambda_r, 1, \dots, 1)$ and $a_{r+1} \leq \dots \leq a_n$ for some integer $r$ satisfying $0\leq r\leq n$.\label{cond:mixed}
\item $a_1\leq \dots \leq a_n$ and $\lambda_i = 1$ for all $i\neq 1$.\label{cond:hinged}
\end{enum*}
\end{theorem}

In particular, the case ${\bf a} = (1, \dots, 1)$ extends Mermin-Murai's work~\cite{MerminMurai2010} and says $\fieldk[X_{\boldsymbol{\lambda}}]/Q_{(1,\dots, 1)}$ is Macaulay-Lex for every composition, not just for compositions satisfying $\lambda_1 \geq \dots \geq \lambda_n$. In contrast, the condition $a_1 \leq \dots \leq a_n$ on Clements-Lindstr\"{o}m rings is necessary for the Macaulay-Lex property to hold. Remarkably, we also get new Macaulay-Lex rings that are `hybrids' of both Clements-Lindstr\"{o}m rings and colored squarefree rings, thereby simultaneously generalizing both the Clements-Lindstr\"{o}m theorem~\cite{ClementsLindstrom1969} and the Frankl-F\"{u}redi-Kalai theorem~\cite{FFK1988}.

Let ${\bf a} = (a_1, \dots, a_n)$ be an $n$-tuple of positive integers. An {\it ${\bf a}$-colored complex} is a simplicial complex $\Delta$ on a non-empty vertex set $V$, together with an ordered partition $(V_1, \dots, V_n)$ of $V$, such that every face $F$ of $\Delta$ satisfies $|F\cap V_i| \leq a_i$. If we treat each $V_i$ as the set of vertices with the `$i$-th color', then every face of $\Delta$ has at most $a_i$ vertices of the $i$-th color. This generalizes the usual notion of colored complexes that many authors use, which coincides with our definition of $(1,\dots, 1)$-colored complexes. The notion of ${\bf a}$-colored multicomplexes can be defined similarly; see Section \ref{sec:GenColoredMulticomplexes} for the definitions of relevant terminology. The Frankl-F\"{u}redi-Kalai theorem~\cite{FFK1988} tells us that for every $(1, \dots, 1)$-colored complex, there exists a ``reverse-lexicographic'' $(1, \dots, 1)$-colored complex with the same $f$-vector. Our next result shows that an analogous statement does not hold for ${\bf a}$-colored complexes or multicomplexes when $n>1$ and ${\bf a} \neq (1, \dots, 1)$.

\begin{theorem}\label{main-thm-f-vectors-colored}
Let ${\bf a} = (a_1, \dots, a_n)$ be an $n$-tuple of positive integers. The following are equivalent:
\begin{enum*}
\item For every ${\bf a}$-colored complex (resp., multicomplex), there exists a reverse-lexicographic ${\bf a}$-colored complex (resp., multicomplex) with the same $f$-vector.
\item Either $n=1$, or ${\bf a} = (1, \dots, 1)$.
\end{enum*}
\end{theorem}

The rest of the paper is organized as follows. In Section \ref{sec:MacLexRings}, we fix our notation and study basic properties of Macaulay-Lex rings that we need. Sections \ref{sec:NonMacLexColoredQuotRings}--\ref{sec:ProofOfCruxResult} deal with the proof of Theorem \ref{main-thm-MacLex}. In Section \ref{sec:GenColoredMulticomplexes}, we study the $f$-vectors of generalized colored multicomplexes with arbitrarily prescribed maximum possible degrees of its variables, and prove Theorem \ref{main-thm-f-vectors-colored} as a special case of Theorem \ref{thm:revlexdoesntcharacterize}. Finally, in Section \ref{sec:FurtherRemarks}, we conclude our paper with further remarks.

\section{Macaulay-Lex rings}\label{sec:MacLexRings}
We begin by fixing our notation and terminology. Let $\mathbb{N}$ and $\mathbb{P}$ denote the non-negative integers and positive integers respectively, and for convenience, let $\overline{\mathbb{N}} = \mathbb{N} \cup \{\infty\}$, $\overline{\mathbb{P}} = \mathbb{P} \cup \{\infty\}$. For $n\in \mathbb{P}$ and ${\bf a} = (a_1, \dots, a_n)$, ${\bf b} = (b_1, \dots, b_n)$ in $\overline{\mathbb{N}}^n$, define $[n] := \{1, \dots, n\}$, let $|{\bf a}| := a_1 + \dots + a_n$, and write ${\bf a} \leq {\bf b}$ if $a_i \leq b_i$ for all $i\in [n]$. We use the convention that ${\bf a} < {\bf b}$ means ${\bf a} \leq {\bf b}$ and ${\bf a} \neq {\bf b}$, and we set $[0] := \emptyset$. For brevity, let $\boldsymbol{\delta}_n$, ${\bf 1}_n$ and $\boldsymbol{\infty}_n$ denote the $n$-tuples $(0,\dots, 0,1)$, $(1, \dots, 1)$ and $(\infty, \dots, \infty)$ respectively. [Note: $|\boldsymbol{\delta}_n| = 1$.]

Throughout this paper, $S:= \fieldk[x_1, \dots, x_n]$ is a standard graded polynomial ring on $n$ variables ($n\in \mathbb{P}$) over a field $\fieldk$, and we fix a linear order $x_1 > \dots > x_n$ on its variables. Given a subset $X \subseteq \{x_1, \dots, x_n\}$ and a monomial $m = x_1^{\alpha_1}\cdots x_n^{\alpha_n}$ in $S$, each integer $\alpha_i$ is called the {\it exponent} of $x_i$ in $m$, and we define $m_X := \prod_{x_i\in X} x_i^{\alpha_i}$. The {\it support} of $m$, denoted by $\supp(m)$, is the set of variables $\{x_i: \alpha_i \neq 0\}$. If $Y$ is a set of elements in $S$, write $\langle Y \rangle$ to mean the ideal generated by $Y$, and write $\Spank(Y)$ to mean the $\fieldk$-vector space spanned by $Y$. For any $\fieldk$-vector space $U \subseteq S$, write $\{U\}$ to mean the set of monomials in $U$.

Let $\Gamma$ be a set of monomials in $S$ of degree $d\in \mathbb{N}$. The {\it lex} (lexicographic) order $\leq_{\ell ex}$ on $\Gamma$ induced by the linear order $x_1 > \dots > x_n$ is given by $x_1^{\alpha_1}\cdots x_n^{\alpha_n} <_{\ell ex} x_1^{\beta_1}\cdots x_n^{\beta_n}$ if and only if $\alpha_i < \beta_i$ for the smallest $i\in [n]$ such that $\alpha_i \neq \beta_i$. Define the {\it revlex} (reverse-lexicographic) order $\leq_{r\ell}$ on $\Gamma$ by $m <_{r\ell} m^{\prime} \Leftrightarrow m >_{\ell ex} m^{\prime}$ for all $m, m^{\prime} \in \Gamma$. Note in particular that $x_1 <_{r\ell} \dots <_{r\ell} x_n$ is the reverse of the fixed linear order on the variables. Given any subset $\Gamma^{\prime} \subseteq \Gamma$, we say $\Gamma^{\prime}$ is a {\it lex-segment} (resp., {\it revlex-segment}) set in $\Gamma$ if $m\in \Gamma^{\prime} \Rightarrow m^{\prime} \in \Gamma^{\prime}$ for all $m^{\prime} >_{\ell ex} m$ (resp., $m^{\prime} >_{r\ell} m$) in $\Gamma$.

Suppose $R = \bigoplus_{d\in \mathbb{N}} R_d := S/M$ for some monomial ideal $M \subseteq S$. Without ambiguity, identify the monomials in $\{R\}$ with the monomials in $\{S\}\backslash\{M\}$ in the natural way, so that previously defined notions (lex order, lex-segment, etc.) make sense on subsets of $\{R_d\}$ for any $d\in \mathbb{N}$. An {\it $R_d$-monomial space} is a $\fieldk$-vector space $A$ spanned by some subset of the monomials in $\{R_d\}$, and we say $A$ is {\it lex-segment} (resp., {\it revlex-segment}) if it is spanned by a lex-segment (resp., revlex-segment) set in $\{R_d\}$. A {\it lex ideal} of $R$ is a (graded) monomial ideal $L = \bigoplus_{d\in \mathbb{N}} L_d$ of $R$ such that each $L_d$ is lex-segment. Given a graded $R$-module $T = \bigoplus_{d\in \mathbb{N}} T_d$, the {\it Hilbert function} of $T$ is the map $H(T,-): \mathbb{N} \to \mathbb{N}$ given by $d \mapsto \dim_{\fieldk} T_d$. Recall that we say $R$ is a {\it Macaulay-Lex ring} (or equivalently, $M$ is a {\it Macaulay-Lex ideal}) if for every graded ideal $I$ of $R$, there exists a lex ideal $L$ of $R$ such that $H(I,d) = H(L,d)$ for all $d\in \mathbb{N}$.

Given $d\in \mathbb{P}$ and a subset $\Gamma \subseteq \{R_d\}$, we say $\partial(\Gamma) := \big\{m\in \{R_{d-1}\}: m\text{ divides }m^{\prime}\text{ for some }m^{\prime} \in \Gamma\big\}$ is the {\it lower shadow} of $\Gamma$. If $A$ is an $R_d$-monomial space, then the {\it lower shadow} of $A$, denoted by $\partial(A)$, is the $R_{d-1}$-monomial space spanned by $\partial(\{A\})$, and the {\it upper shadow} of $A$ is the $R_{d+1}$-monomial space $R_1A$ spanned by the monomials of the form $x_im$, where $x_i \in R_1$ and $m\in \{A\}$. For convenience, let $\Lex_{R_d}(A)$ and $\Revlex_{R_d}(A)$ denote the unique lex-segment and revlex-segment $R_d$-monomial spaces of dimension $|A|$ respectively, where $|A| := \dim_{\fieldk} (A)$ is the dimension of $A$ as a $\fieldk$-vector space.

\begin{theorem}\label{thm:charMacLexRings}
The following are equivalent:
\begin{enum*}
\item $R$ is a Macaulay-Lex ring.\label{list:charMacLex-MacLex}
\item $|R_1(\Lex_{R_d}(A))| \leq |R_1A|$ for all $d\in \mathbb{N}$ and every $R_d$-monomial space $A$.\label{list:charMacLex-uppershad}
\item $\partial(\Revlex_{R_{d+1}}(A)) \subseteq \Revlex_{R_d}(\partial(A))$ for all $d\in \mathbb{N}$ and every $R_{d+1}$-monomial space $A$.\label{list:charMacLex-lowershadContain}
\item\label{list:charMacLex-lowershadComplete} For all $d\in \mathbb{N}$, the following two conditions hold:
\begin{enum*}
\item\label{list:charMacLex-lowershadComplete-ineq} $|\partial(\Revlex_{R_{d+1}}(A))| \leq |\partial(A)|$ for every $R_{d+1}$-monomial space $A$; and
\item\label{list:charMacLex-lowershadComplete-lowershad} the lower shadow of every revlex-segment $R_{d+1}$-monomial space is revlex-segment.
\end{enum*}
\end{enum*}
\end{theorem}

\begin{proof}
The equivalence \ref{list:charMacLex-MacLex} $\Leftrightarrow$ \ref{list:charMacLex-uppershad} is straightforward (see, e.g., \cite[Thm. 2.1.7]{Shakin2005:MonomialIdeals}) and was first proven by Macaulay~\cite{Macaulay1927} for the case $R = S$. The equivalence \ref{list:charMacLex-MacLex} $\Leftrightarrow$ \ref{list:charMacLex-lowershadContain} was first proven by Clements and Lindstr\"{o}m~\cite{ClementsLindstrom1969} for Clements-Lindstr\"{o}m rings, and a proof for arbitrary Macaulay-Lex rings was given by Shakin~\cite[Thm. 2.7]{Shakin2001:GeneralizationMacaulayThm}. Finally, Engel~\cite[Prop. 8.1.1]{book:EngelSpernerTheory} proved \ref{list:charMacLex-lowershadContain} $\Leftrightarrow$ \ref{list:charMacLex-lowershadComplete}.
\end{proof}

\begin{remark}
The converses to the two implications \ref{list:charMacLex-MacLex} $\Rightarrow$ \ref{list:charMacLex-lowershadComplete-ineq} and \ref{list:charMacLex-MacLex} $\Rightarrow$ \ref{list:charMacLex-lowershadComplete-lowershad} in Theorem \ref{thm:charMacLexRings} are not true. For example, as observed by Shakin~\cite[Example 2.8]{Shakin2001:GeneralizationMacaulayThm}, the quotient ring $R = \fieldk[x_1, x_2]/\langle x_1^3, x_1x_2^2, x_2^3 \rangle$ satisfies \ref{list:charMacLex-lowershadComplete-ineq}, yet this ring is not Macaulay-Lex. Also, we will later show that colored quotient rings satisfy \ref{list:charMacLex-lowershadComplete-lowershad} (see Proposition \ref{prop:ShadRevlexSegment}), yet by Theorem \ref{main-thm-MacLex}, not every colored quotient ring is Macaulay-Lex. In contrast, the upper shadow of every lex-segment $R_d$-monomial space (for all $d\in \mathbb{N}$) is lex-segment; see, e.g., \cite[Lem. 2.1]{Shakin2001:GeneralizationMacaulayThm} for a proof.
\end{remark}

Given an ideal $I$ of $S$, a {\it lex-plus-$I$ ideal} of $S$ is an ideal $L^{\prime}$ that can be written as $L^{\prime} = L + I$ for some lex ideal $L$ of $S$. In particular, if $P := \langle x_1^{a_1}, \dots, x_n^{a_n} \rangle$ for $2\leq a_1 \leq \dots \leq a_n \leq \infty$, then lex-plus-$P$ ideals are called {\it lex-plus-powers ideals}, and they were first introduced by Evans; see \cite{FranciscoRichert2007:LPPIdeals}.
\begin{theorem}\label{thm:MacLexProperties}
Let $M$ be a Macaulay-Lex ideal of $S$.
\begin{enum*}
\item\label{thm-part:lex-plus-M} Every lex-plus-$M$ ideal of $S$ is Macaulay-Lex. In particular, every lex-plus-powers ideal of $S$ is a Macaulay-Lex ideal.
\item\label{thm-part:minDegGenerator} If $M$ is non-zero and $d>0$ is the minimal degree of the generators of $M$, then there exists some $i\in [n]$ such that $x_1^{d-1}x_i \in M$.
\item\label{thm-part:truncatedMacLexIdeal} If $n\geq 2$ and $d\in \mathbb{N}$, then the ideal $\big\langle\big\{m\in \{\fieldk[x_2, \dots, x_n]\}: mx_1^d \in M\big\}\big\rangle$ contained in $\fieldk[x_2, \dots, x_n]$ is Macaulay-Lex.
\item\label{thm-part:extendNewVarMacLex} If $y$ is an indeterminate, then $(S[y])M$ as an ideal of the polynomial ring $S[y]$ is Macaulay-Lex with respect to the linear order $x_1 > \dots > x_n > y$.
\end{enum*}
\end{theorem}

\begin{proof}
Statement \ref{thm-part:lex-plus-M} was independently proven by Shakin~\cite[Thm. 2.7.2]{Shakin2005:MonomialIdeals} and Mermin-Peeva~\cite[Thm. 5.1]{MerminPeeva2006:LexifyingIdeals}. Statements \ref{thm-part:minDegGenerator} and \ref{thm-part:truncatedMacLexIdeal} were first proven by Shakin; see \cite[Thm. 2.2.4]{Shakin2005:MonomialIdeals} and \cite[Prop. 2.5.2]{Shakin2005:MonomialIdeals} respectively. Statement \ref{thm-part:extendNewVarMacLex} was independently proven by Shakin~\cite[Thm. 2.5.1]{Shakin2005:MonomialIdeals} and Mermin-Peeva~\cite[Thm. 4.1]{MerminPeeva2006:LexifyingIdeals}.
\end{proof}

In this paper, we always reserve the notation $W = \bigoplus_{d\in \mathbb{N}} W_d$ to mean a colored quotient ring of type ${\bf a} \in \overline{\mathbb{P}}^n$ and composition $\boldsymbol{\lambda} \in \mathbb{P}^n$ (as defined in Section \ref{sec:Intro}). In particular, $S \cong \fieldk[X_{{\bf 1}_n}]/Q_{\boldsymbol{\infty}_n}$ is a colored quotient ring of type ${\bf a} = \boldsymbol{\infty}_n$ and composition $\boldsymbol{\lambda} = {\bf 1}_n$. Recall that the lower shadow of an arbitrary revlex-segment $R_{d+1}$-monomial space is not necessarily revlex-segment. Nevertheless, we have the following:

\begin{proposition}\label{prop:ShadRevlexSegment}
Let $d\in \mathbb{N}$. If $A$ is a revlex-segment $W_{d+1}$-monomial space, then $\partial(A)$ is revlex-segment.
\end{proposition}

\begin{proof}
For convenience, relabel the variables in $X_{\boldsymbol{\lambda}}$ by $y_1 < y_2 < y_3 < \dots$, so that the linear order on $X_{\boldsymbol{\lambda}}$ is preserved, i.e. $y_1 = x_{n,1}$, $y_2 = x_{n-1,1}$, $y_3 = x_{n-2,1}$, etc. Suppose $m$ is a monomial in $\partial(A)$. Then there is a $y_t$ such that $my_t$ is a monomial in $A$. Without loss of generality, choose the largest possible $t$ (i.e. the smallest possible $y_t$ in the revlex order). Next, choose an arbitrary monomial $u$ in $W_d$ such that $u >_{r\ell} m$, and write $m = y_1^{\alpha_1}\cdots y_N^{\alpha_N}$, $u = y_1^{\beta_1}\cdots y_N^{\beta_N}$, where $N = \max\{i: y_i\text{ divides }m\text{ or }y_i\text{ divides }u\}$.

To show that $\partial(A)$ is revlex-segment, it suffices to show that $u \in \partial(A)$. If there is some $i\in [t]$ such that $uy_i \neq 0$, then $y_i \geq_{r\ell} y_t$ and $u >_{r\ell} m$ together yield $uy_i >_{r\ell} my_t$, so since $A$ is revlex-segment, we get $uy_i \in A$, which implies $u\in \partial(A)$, and we are done. Thus, we can assume every $i\in [t]$ satisfies $uy_i = 0$.

Note that $my_t \neq 0$ yields $\deg(my_t) \leq |{\bf a}|$, hence $\deg(u) = \deg(m) \leq |{\bf a}| - 1$, and $uy_j \neq 0$ for some $j>t$. Choose the smallest possible $j$, so that $uy_i = 0$ for all $i\in [j-1]$. Note that $j\leq n$, since otherwise $uy_i = 0$ for all $i\in [n]$, which implies $\deg(u_{X_i}) = a_i < \infty$ for all $i\in [n]$. This would force $\deg(m) = \deg(u) = |{\bf a}|$ and $mx = 0$ for all $x\in X_{\boldsymbol{\lambda}}$, which contradicts the existence of $y_t$. Also, notice that if we can show $uy_j \geq_{r\ell} my_t$, then since $A$ is revlex-segment, we get $uy_j\in A$, i.e. $u\in \partial(A)$, and we are done.

Suppose instead $uy_j <_{r\ell} my_t$. Let $\ell\in [N]$ be the largest integer such that $\alpha_{\ell} \neq \beta_{\ell}$, and note that $u >_{r\ell} m$ implies $\beta_{\ell} < \alpha_{\ell}$. By the definition of $\ell$, it follows from $uy_j <_{r\ell} my_t$ and $j>t$ that $j\geq \ell$. If $j> \ell$, then we get $\ell\in [n]$, so since $\beta_{\ell} < \alpha_{\ell}$, the definition of $\ell$ implies $\deg(u_{X_{n+1-\ell}}) < \deg(m_{X_{n+1-\ell}})$, thus $uy_{\ell} \neq 0$, which contradicts the minimality of $j$. Consequently, $j = \ell$. This forces $\beta_{\ell} = \alpha_{\ell} - 1$, since otherwise $\beta_{\ell} < \alpha_{\ell} - 1$ would imply the contradiction $uy_j = uy_{\ell} >_{r\ell} my_t$ that follows from $t<j=\ell$ and the definition of $\ell$.

Now, $uy_{\ell} = uy_j <_{r\ell} my_t$, and the exponents of $y_{\ell}$ in both $uy_{\ell}$ and $my_t$ are equal, thus there is some $\ell^{\prime} \in [\ell-1]$ such that the exponent of $y_{\ell^{\prime}}$ in $uy_{\ell}$ is strictly greater than the exponent of $y_{\ell^{\prime}}$ in $my_t$. In particular, this means $\alpha_{\ell^{\prime}} < \beta_{\ell^{\prime}}$. Note that $uy_i = 0$ for all $i\in [\ell-1]$, so it follows from $\ell \leq n$ and the definition of $W$ that $\deg(u_{X_{n+1-i}}) = a_{n+1-i} < \infty$ for all $i\in [\ell-1]$, thus $\alpha_i \leq \beta_i$ for all $i\in [\ell-1]$. Since $\deg(uy_{\ell}) = \deg(my_t)$ and $\beta_{\ell} = \alpha_{\ell}-1$, we get $\sum_{i\in [\ell-1]} \beta_i = \big(\sum_{i\in [\ell-1]} \alpha_i\big)+1$, so $\alpha_{\ell^{\prime}} < \beta_{\ell^{\prime}}$ implies $\alpha_i = \beta_i$ for all $i\in [\ell-1]\backslash\{\ell^{\prime}\}$, and $\beta_{\ell^{\prime}} = \alpha_{\ell^{\prime}}-1$. Note that $\ell^{\prime} \neq t$, since otherwise we would get $uy_{\ell} = my_t$. Thus $\deg(u_{X_{n+1-\ell^{\prime}}}) < \deg(m_{X_{n+1-\ell^{\prime}}})$, which implies $uy_{\ell^{\prime}} \neq 0$, therefore contradicting the minimality of $j$.
\end{proof}

By combining with Theorem \ref{thm:charMacLexRings}, we get the following useful corollary.

\begin{corollary}\label{cor:coloredMacLexChar}
$W$ is Macaulay-Lex if and only if $|\partial(\Revlex_{W_{d+1}}(A))| \leq |\partial(A)|$ for all $d\in \mathbb{N}$ and every $W_{d+1}$-monomial space $A$.
\end{corollary}

\section{Non-Macaulay-Lex colored quotient rings}\label{sec:NonMacLexColoredQuotRings}
Let $W$ be a colored quotient ring of type ${\bf a} = (a_1, \dots, a_n)$ and composition $\boldsymbol{\lambda} = (\lambda_1, \dots, \lambda_n)$. If there exists some integer $r$ satisfying $0\leq r\leq n$ such that $a_i = 1$ for all $i\in [r]$, $\lambda_i = 1$ for all $i\in [n]\backslash [r]$, and $a_{r+1}\leq \dots \leq a_n$, then we say $W$ is {\it mixed}. If $a_1\leq \dots \leq a_n$ and $\lambda_i = 1$ for all $i\neq 1$, then we say $W$ is {\it hinged}. Note that a Clements-Lindstr\"{o}m ring is both mixed and hinged, while a colored squarefree ring is mixed. Theorem \ref{main-thm-MacLex} thus asserts that a colored quotient ring is Macaulay-Lex if and only if it is either mixed or hinged. Our main result for this section is to show that if $W$ is Macaulay-Lex, then it must be either mixed or hinged. We start by proving a useful necessary condition for $W$ to be Macaulay-Lex.

\begin{proposition}\label{prop:necessaryCondOnType}
If $W$ is Macaulay-Lex, then $a_1 \leq \dots \leq a_n$.
\end{proposition}

\begin{proof}
Assume $n>1$, and suppose on the contrary that $a_t > a_{t+1}$ for some $t\in [n-1]$. Define the lex ideal $L := \langle \{x\in X_{\boldsymbol{\lambda}}: x > x_{t,1}\} \rangle \subseteq \fieldk[X_{\boldsymbol{\lambda}}]$, and note that $Q_{{\bf a}}$ is Macaulay-Lex by assumption, hence by Theorem \ref{thm:MacLexProperties}\ref{thm-part:lex-plus-M}, the ideal $M := Q_{{\bf a}} + L$ is a Macaulay-Lex ideal of $\fieldk[X_{\boldsymbol{\lambda}}]$. Since
\begin{equation*}
\fieldk[X_{\boldsymbol{\lambda}}]/M \cong \fieldk[x_{t,1}, x_{(t+1),1}, \dots, x_{n,1}]/\langle x_{t,1}^{a_t+1}, x_{(t+1),1}^{a_{(t+1)}+1}, \dots, x_{n,1}^{a_n+1}\rangle,
\end{equation*}
it follows that $\langle x_{t,1}^{a_t+1}, x_{(t+1),1}^{a_{(t+1)}+1}, \dots, x_{n,1}^{a_n+1}\rangle$ is a Macaulay-Lex ideal of $\fieldk[x_{t,1}, x_{(t+1),1}, \dots, x_{n,1}]$, whose generators have minimal degree $\leq a_t$, which then contradicts Theorem \ref{thm:MacLexProperties}\ref{thm-part:minDegGenerator}.
\end{proof}

\begin{remark}\label{remark:caseCompositionAllOnes}
Proposition \ref{prop:necessaryCondOnType} was previously known for the case $\boldsymbol{\lambda} = {\bf 1}_n$; see \cite[Example 2.7.4]{Shakin2005:MonomialIdeals}. By combining with the Clements-Lindstr\"{o}m theorem~\cite{ClementsLindstrom1969}, we conclude that a colored quotient ring of composition ${\bf 1}_n$ is Macaulay-Lex if and only if $a_1 \leq \dots \leq a_n$.
\end{remark}

\begin{theorem}\label{thm:MacLex=>MixedOrHinged}
If $W$ is Macaulay-Lex, then $W$ is either mixed or hinged.
\end{theorem}

\begin{proof}
Let $s := \max\{i\in [n]: \lambda_i \geq 2\}$. By Proposition \ref{prop:necessaryCondOnType}, it suffices to show that if $s\neq 1$, then $a_i = 1$ for all $i\in [s]$. Define $M^{\prime} := \langle x_{1,1}^{a_1+1}, x_{2,1}^{a_2+1}, \dots, x_{n,1}^{a_n+1}\rangle + \langle x_{s,1}, x_{s,2} \rangle^{a_s+1} \subseteq \fieldk[x_{s,2}, x_{1,1}, x_{2,1}, \dots, x_{n,1}]$. Note that $Q_{{\bf a}}$ is Macaulay-Lex, while $L := \langle \{x\in X_{\boldsymbol{\lambda}}: x > x_{s,2} \}\rangle \subseteq \fieldk[X_{\boldsymbol{\lambda}}]$ is a lex ideal, so $M := Q_{{\bf a}} + L$ is Macaulay-Lex by Theorem \ref{thm:MacLexProperties}\ref{thm-part:lex-plus-M}. Since $\fieldk[X_{\boldsymbol{\lambda}}]/M \cong \fieldk[x_{s,2}, x_{1,1}, x_{2,1}, \dots, x_{n,1}]/M^{\prime}$, it follows that $M^{\prime}$ is Macaulay-Lex.

Next, define the ideal $J = \langle x_{1,1}^{a_1+1}, x_{2,1}^{a_2+1}, \dots, x_{n,1}^{a_n+1}\rangle + \langle x_{s,1}^2\rangle \subseteq \fieldk[x_{1,1}, x_{2,1}, \dots, x_{n,1}]$, which we note is Macaulay-Lex by Theorem \ref{thm:MacLexProperties}\ref{thm-part:truncatedMacLexIdeal}. This means $\fieldk[x_{1,1}, x_{2,1}, \dots, x_{n,1}]/J$ is a Macaulay-Lex colored quotient ring of type $(a_1, \dots, a_{s-1}, 1, a_{s+1}, \dots, a_n)$, so Proposition \ref{prop:necessaryCondOnType} yields $a_i = 1$ for all $i\in [s-1]$. Consequently, $M^{\prime} = \langle x_{1,1}^2, \dots, x_{(s-1),1}^2\rangle + \langle x_{s,1}^{a_s+1}, \dots, x_{n,1}^{a_n+1} \rangle + \langle x_{s,1}, x_{s,2} \rangle^{a_s+1}$. Now, if $s>1$, then $M^{\prime}$ has generators of minimal degree $2$, and Theorem \ref{thm:MacLexProperties}\ref{thm-part:minDegGenerator} says there is some $x\in \{x_{1,1}, \dots, x_{n,1}\} \cup \{x_{s,2}\}$ such that $x_{s,2}x \in M^{\prime}$, thereby forcing $x\in \{x_{s,1}, x_{s,2}\}$ and hence $a_i = 1$ for all $i\in [s]$.
\end{proof}

\section{Mixed or hinged colored quotient rings}\label{sec:MixedHingedColoredQuotRings}
In this section, our goal is to prove that if $W$ is either mixed or hinged, then $W$ is Macaulay-Lex, thereby completing the proof of Theorem \ref{main-thm-MacLex}. We begin by introducing the notion of quasi-compression and describing our proof strategy. In particular, we state a key result (Theorem \ref{thm:CruxPart}) that is necessary for our proof strategy to work. The proof of Theorem \ref{thm:CruxPart} requires some preparation, so we postpone it to Section \ref{sec:ProofOfCruxResult}. Using Theorem \ref{thm:CruxPart}, we then complete the proof of Theorem \ref{main-thm-MacLex}.

Suppose $n>1$. For each $t\in [n]$, let $\widehat{S}^{\langle t\rangle} := \fieldk[X_{\boldsymbol{\lambda}}-X_t]$ be a polynomial ring on the set of variables $X_{\boldsymbol{\lambda}}-X_t$ over a field $\fieldk$, and fix the linear order on $X_{\boldsymbol{\lambda}}-X_t$ induced by the order on $X_{\boldsymbol{\lambda}}$ given in Definition \ref{defn-coloredQuot}. Define the colored quotient ring $\widehat{W}^{\langle t \rangle} = \bigoplus_{d\in \mathbb{N}} \widehat{W}^{\langle t\rangle}_d := \widehat{S}^{\langle t \rangle}/\widehat{Q}_{{\bf a}}^{\langle t\rangle}$, where $\widehat{Q}_{{\bf a}}^{\langle t \rangle} := \sum_{i\in [n]\backslash\{t\}} \langle X_t \rangle^{a_i+1}$ is an ideal of $\widehat{S}^{\langle t \rangle}$. Note that if $W$ is mixed (resp., hinged), then $\widehat{W}^{\langle t \rangle}$ is mixed (resp., hinged).

Let $A$ be a $W_d$-monomial space for some $d\in \mathbb{N}$. Define $\|A\|_{({\bf a}, \boldsymbol{\lambda})} := \sum_{m\in \{A\}} \|m\|_{({\bf a}, \boldsymbol{\lambda})}$, where $\|m\|_{({\bf a}, \boldsymbol{\lambda})} := \text{\mbox{$\big|\big\{m^{\prime}\in \{W_{\deg(m)}\}: m^{\prime} \geq_{r\ell} m\big\}\big|$}}$ for each $m\in \{W\}$. If $q$ is a monomial in $W$ such that $\supp(q) \subseteq X_t$ for some $t\in [n]$, then define $A[t; q] := \Spank(\{ m\in \{A\}: m_{X_t} = q\})$ and $A[t;q]/q := \Spank(\{\tfrac{m}{q}: m\in \{A[t;q]\}\})$. Also, let $q\cdot A^{\prime} := \Spank(\{q\cdot m: m\in \{A^{\prime}\}\})$ for any $\widehat{W}_{d-\deg(q)}^{\langle t\rangle}$-monomial space $A^{\prime}$. In particular, $A[t; q] = q\cdot (A[t; q]/q)$, and we can write $\{A\} = \bigsqcup_q \{A[t; q]\}$, where the disjoint union is over all monomials $q$ in $W$ satisfying $\supp(q) \subseteq X_t$ for some fixed $t\in [n]$.

\begin{definition} If $n>1$ and $t\in [n]$, then define the operation $\mathcal{C}_t$ on $W_d$-monomial spaces by
\begin{equation*}
\mathcal{C}_t(A) := \!\!\!\!\bigoplus_{\substack{q\in \{W\!\}\\ \supp(q)\subseteq X_t}}\!\!\!\! q\cdot\Revlex_{\widehat{W}_{d-\deg(q)}^{\langle t\rangle}}\!\!\!\!\big(A[t; q]/q\big) =\!\! \bigoplus_{\substack{q\in \{W\!\}\\ \supp(q)\subseteq X_t}} \!\!\!\!\Spank\!\!\Big(\Big\{ qm: m\in \big\{\Revlex_{\widehat{W}_{d-\deg(q)}^{\langle t\rangle}} \!\!\!\!\big(A[t; q]/q\big)\big\}\Big\}\Big).
\end{equation*}
\end{definition}

It is clear from definition that $\mathcal{C}_t(A)$ is a $W_d$-monomial space satisfying $\|\mathcal{C}_t(A)\|_{({\bf a}, \boldsymbol{\lambda})} \leq \|A\|_{({\bf a}, \boldsymbol{\lambda})}$ and $|\mathcal{C}_t(A)| = |A|$. Assuming $n>1$, we say $A$ is {\it $(X_{\boldsymbol{\lambda}} - X_t)$-compressed} if $\mathcal{C}_t(A) = A$, and we say $A$ is {\it quasi-compressed} if $A$ is $(X_{\boldsymbol{\lambda}} - X_t)$-compressed for all $t\in [n]$. This notion of `$(X_{\boldsymbol{\lambda}} - X_t)$-compressed' agrees with the notion of $\mathcal{A}$-compression introduced by Mermin~\cite{Mermin2008:CompressedIdeals} (for $\mathcal{A} = X_{\boldsymbol{\lambda}} - X_t$). Also, the notion of quasi-compression coincides with the usual notion of compression in the case $\boldsymbol{\lambda} = {\bf 1}_n$ (and $n>1$); cf. \cite{MerminPeeva2006:LexifyingIdeals}. However in general, a quasi-compressed $W_d$-monomial space is not necessarily compressed. For the special case $n=1$, define $\mathcal{C}_1(A) := A$, and define $A$ to be always quasi-compressed.

\begin{lemma}\label{lemma:quasicompressionfiniteseq}
Let $d\in \mathbb{N}$. For every $W_d$-monomial space $A$, there is a finite sequence $t_1, \dots, t_r$ of integers in $[n]$ such that $\mathcal{C}_{t_r}(\mathcal{C}_{t_{r-1}}(\cdots \mathcal{C}_1(A)))$ is quasi-compressed. Furthermore, we can choose $t_1, \dots, t_r$ so that $\|\mathcal{C}_{t_{i+1}}(\mathcal{C}_{t_i}(\cdots \mathcal{C}_1(A)))\|_{({\bf a}, \boldsymbol{\lambda})} < \|\mathcal{C}_{t_i}(\cdots \mathcal{C}_1(A))\|_{({\bf a}, \boldsymbol{\lambda})}$ for every pair of consecutive terms $t_i, t_{i+1}$.
\end{lemma}

\begin{proof}
Clearly $\|\mathcal{C}_t(A)\|_{({\bf a}, \boldsymbol{\lambda})} \leq \|A\|_{({\bf a}, \boldsymbol{\lambda})}$ for each $t\in [n]$, with equality holding if and only if $\mathcal{C}_t(A) = A$.
\end{proof}

To prove that a given colored quotient ring $W$ is Macaulay-Lex, Corollary \ref{cor:coloredMacLexChar} tells us it suffices to show that for any $W_d$-monomial space $A$, there exists a finite sequence $A_0, A_1, \dots, A_N$ of $W_d$-monomial spaces starting with $A_0 = A$ and ending with $A_N = \Revlex_{W_d}(A)$, such that $|\partial(A_i)| \leq |\partial(A_{i-1})|$ for all $i\in [N]$. By first defining the term $A_0 = A$, we construct such a sequence algorithmically as follows:

\begin{center}
\begin{algorithmic}
\State \underline{{\bf Step $i$:}} $A_{i-1}$ is already defined; $A_i$ is not yet defined.
\If {$A_{i-1}$ is not quasi-compressed}
    \State Choose some $t\in [n]$ such that $\|A_{i-1}\|_{({\bf a}, \boldsymbol{\lambda})} > \|\mathcal{C}_t(A_{i-1})\|_{({\bf a}, \boldsymbol{\lambda})}$. Set $A_i = \mathcal{C}_t(A_{i-1})$.
\Else {\ \ (i.e. $A_{i-1}$ is quasi-compressed)}
    \If {there exists a $W_d$-monomial space $A^{\prime}$ such that $|A^{\prime}| = |A_{i-1}|$, $\|A^{\prime}\|_{({\bf a}, \boldsymbol{\lambda})} < \|A_{i-1}\|_{({\bf a}, \boldsymbol{\lambda})}$, and $|\partial(A^{\prime})| \leq |\partial(A_{i-1})|$}
        \State Choose any such $A^{\prime}$ and set $A_i = A^{\prime}$.
    \Else
        \State Terminate algorithm.
    \EndIf
\EndIf
\end{algorithmic}
\end{center}
This algorithm is well-defined by Lemma \ref{lemma:quasicompressionfiniteseq}. Also, note that $\|A_0\|_{({\bf a}, \boldsymbol{\lambda})} > \|A_1\|_{({\bf a}, \boldsymbol{\lambda})} > \ldots > \|A_N\|_{({\bf a}, \boldsymbol{\lambda})}$ by construction. To guarantee $A_N = \Revlex_{W_d}(A)$ and $|\partial(A_i)| \leq |\partial(A_{i-1})|$ for all $i\in [N]$, we need to show the following:
\begin{enum}
\item Every $W_d$-monomial space $A$ satisfies $|\partial(\mathcal{C}_t(A))| \leq |\partial(A)|$ for all $t\in [n]$.\label{algorithm-necessary-within-sequence}
\item If $A$ is a quasi-compressed $W_d$-monomial space, then either $A$ is revlex-segment, or there exists some $W_d$-monomial space $A^{\prime}$ such that $|A^{\prime}| = |A|$, $\|A^{\prime}\|_{({\bf a}, \boldsymbol{\lambda})} < \|A\|_{({\bf a}, \boldsymbol{\lambda})}$, and $|\partial(A^{\prime})| \leq |\partial(A)|$.\label{algorithm-necessary-between-sequences}
\end{enum}
In general, we can find pairs $({\bf a}, \boldsymbol{\lambda})$ such that at least one of these two statements does not hold. This is expected, since Theorem \ref{thm:MacLex=>MixedOrHinged} says $W$ is Macaulay-Lex only if $W$ is either mixed or hinged. In the case when $W$ is either mixed or hinged, we have the following key result:

\begin{theorem}\label{thm:CruxPart}
Let $d\in \mathbb{P}$, $n>1$, and ${\bf a} \in \mathbb{P}^n$. Suppose $W$ is either mixed or hinged. If $A$ is a non-empty quasi-compressed $W_d$-monomial space that is not revlex-segment, then there exists some $W_d$-monomial space $A^{\prime}$ satisfying $|A^{\prime}| = |A|$, $\|A^{\prime}\|_{({\bf a}, \boldsymbol{\lambda})} < \|A\|_{({\bf a}, \boldsymbol{\lambda})}$, and $|\partial(A^{\prime})| \leq |\partial(A)|$.
\end{theorem}

Remarkably, the condition that $W$ is either mixed or hinged is not just necessary, but crucial as well. Our proof of Theorem \ref{thm:CruxPart} is purely combinatorial and uses the combinatorial structure of mixed and hinged colored quotient rings in a fundamental way. This proof requires some preparation, so we postpone it to Section \ref{sec:ProofOfCruxResult}. Note that Theorem \ref{thm:CruxPart}, together with the preceding discussion, implies the following result.

\begin{corollary}\label{cor:KeyIntermediateStep}
Let $d\in \mathbb{P}$, $n>1$, and ${\bf a} \in \mathbb{P}^n$. Suppose $W$ is either mixed or hinged. If $|\partial(\mathcal{C}_t(A))| \leq |\partial(A)|$ for all $t\in [n]$, $d\in \mathbb{P}$, and every $W_d$-monomial space $A$, then $W$ is Macaulay-Lex.
\end{corollary}

Finally, we prove the following theorem, after which we complete the proof of Theorem \ref{main-thm-MacLex}.

\begin{theorem}\label{thm:color-comp}
Let $n,d\in \mathbb{P}$, ${\bf a}\in \mathbb{P}^n$, suppose $W$ is either mixed or hinged, and let $A$ be a $W_d$-monomial space. Then the following two statements hold:
\begin{enum*}
\item\label{item:actual-result} $|\partial(\Revlex_{W_d}(A))| \leq |\partial(A)|$.
\item\label{item:color-comp} $|\partial(\mathcal{C}_t(A))| \leq |\partial(A)|$ for all $t\in [n]$.
\end{enum*}
\end{theorem}

\begin{proof}
We will prove \ref{item:actual-result} and \ref{item:color-comp} simultaneously by induction on $n$. For the base case $n=1$, \ref{item:actual-result} is an easy consequence of Macaulay's theorem (cf. Lemma \ref{lemma:truncated-n=1-case}), while \ref{item:color-comp} is trivially true. Assume $n>1$, fix some $t\in [n]$, and let $\hat{A} := \partial(A)$. For brevity, let $\Gamma$ be the set of all monomials $q$ in $W$ satisfying $\supp(q) \subseteq X_t$. Next, define $B := \Spank(\{W_d\}\backslash \{\mathcal{C}_t(A)\})$ and $\hat{B} := \Spank(\{W_{d-1}\}\backslash \{\mathcal{C}_t(\hat{A})\})$.

We claim that $x\cdot \hat{B} \subseteq B$ for all $x\in X_{\boldsymbol{\lambda}}$. If this claim is true, then $\partial(\mathcal{C}_t(A)) \subseteq \mathcal{C}_t(\hat{A})$, which implies $|\partial(\mathcal{C}_t(A))| \leq |\mathcal{C}_t(\hat{A})| = |\hat{A}| = |\partial(A)|$, i.e. \ref{item:color-comp} is true, thus \ref{item:actual-result} follows from Corollary \ref{cor:KeyIntermediateStep}, and we are done. For each $q\in \Gamma$ of degree $r$, define
\begin{align*}
D^{[q]} &:= \Spank\big(\{\widehat{W}_{d-r}^{\langle t\rangle}\!\} - \{\Revlex_{\widehat{W}_{d-r}^{\langle t\rangle}}\!\!(A[t;q]/q)\}\big);\\
\hat{D}^{[q]} &:= \Spank\big(\{\widehat{W}_{d-r-1}^{\langle t\rangle}\!\} - \{\Revlex_{\widehat{W}_{d-r-1}^{\langle t\rangle}}\!\!(\hat{A}[t;q]/q)\}\big);\\
I^{[q]} &:= \Spank\Big(\Big(\bigsqcup_{j\geq d-r-1} \!\!\!\{\widehat{W}_j^{\langle t\rangle}\}\Big) - \big((A[t;q]/q) \sqcup (\hat{A}[t;q]/q)\big)\Big).
\end{align*}
Note that $B = \bigoplus_{q\in \Gamma} (q\cdot D^{[q]})$ and $\hat{B} = \bigoplus_{q\in \Gamma} (q\cdot \hat{D}^{[q]})$. Consequently, to prove the claim, it suffices to show that $xq\cdot \hat{D}^{[q]} \subseteq B$ for all $x\in X_{\boldsymbol{\lambda}}$ and all $q\in \Gamma$.

Given $q\in \Gamma$, note that $I^{[q]} =  \bigoplus_{j\in \mathbb{N}} I_j^{[q]}$ is a graded monomial ideal of $\widehat{W}^{\langle t\rangle}$. By induction hypothesis, $|\partial(\Revlex_{\widehat{W}_j^{\langle t\rangle}}(A^{\prime}))| \leq |\partial(A^{\prime})|$ for all $j\in \mathbb{P}$ and every $\widehat{W}_j^{\langle t\rangle}$-monomial space $A^{\prime}$, so $\widehat{W}^{\langle t\rangle}$ is Macaulay-Lex by Corollary \ref{cor:coloredMacLexChar}, i.e. there is a lex ideal $L^{[q]} = \bigoplus_{j\in \mathbb{N}} L_j^{[q]}$ in $\widehat{W}^{\langle t\rangle}$ with the same Hilbert function as $I^{[q]}$. Let $r = \deg(q)$, and note that $L^{[q]}_{d-r-1} = \hat{D}^{[q]}$ and $L^{[q]}_{d-r} = D^{[q]}$ by definition. Thus if $x\in X_{\boldsymbol{\lambda}} - X_t$, then $xq\cdot \hat{D}^{[q]} \subseteq q\cdot D^{[q]} \subseteq B$. Suppose instead $x\in X_t$, and assume without loss of generality that $xq\neq 0$. Since $I := \bigoplus_{q\in \Gamma} (q\cdot I^{[q]})$ is a monomial ideal of $W$, we have $x(q\cdot I^{[q]}) \subseteq xq\cdot I^{[xq]}$, thus $I^{[q]} \subseteq I^{[xq]}$, and in particular, $|\hat{D}^{[q]}| = |I^{[q]}_{d-r-1}| \leq |I^{[xq]}_{d-r-1}| = |D^{[xq]}|$. Now, $\hat{D}^{[q]}$ and $D^{[xq]}$ are lex-segment $\widehat{W}_{d-r-1}^{\langle t\rangle}$-monomial spaces by construction, so $\hat{D}^{[q]} \subseteq D^{[xq]}$, therefore $xq\cdot \hat{D}^{[q]} \subseteq xq\cdot D^{[xq]} \subseteq B$.
\end{proof}

\begin{paragraph}{\it Proof of Theorem \ref{main-thm-MacLex}.}
In view of Corollary \ref{cor:coloredMacLexChar}, Theorem \ref{thm:MacLex=>MixedOrHinged} and Theorem \ref{thm:color-comp}, we are left to show that if $n>1$, ${\bf a}\in \overline{\mathbb{P}}^n$ satisfies $a_i = \infty$ for some $i\in [n]$, and $W$ is either mixed or hinged, then $W$ is Macaulay-Lex. (Note that the case $n=1$, ${\bf a} = (\infty)$ is identically Macaulay's theorem.) Given such ${\bf a}\in \overline{\mathbb{P}}^n$ and $W$, fix some $d\in \mathbb{P}$ and define ${\bf b} = (b_1, \dots, b_n) \in \mathbb{P}^n$ by $b_i = \min\{a_i, d+1\}$ for each $i\in [n]$. Since $W$ is either mixed or hinged, the colored quotient ring $\fieldk[X_{\boldsymbol{\lambda}}]/Q_{{\bf b}}$ is also either mixed or hinged, so we know $\fieldk[X_{\boldsymbol{\lambda}}]/Q_{{\bf b}}$ is Macaulay-Lex. Since $\{W_d\} = \{(\fieldk[X_{\boldsymbol{\lambda}}]/Q_{{\bf b}})_d\}$, Corollary \ref{cor:coloredMacLexChar} yields $|\partial(\Revlex_{W_d}(A))| \leq |\partial(A)|$ for every $W_d$-monomial space $A$. This is true for all $d\in \mathbb{P}$, therefore $W$ is Macaulay-Lex by Corollary \ref{cor:coloredMacLexChar}.
\end{paragraph}

\section{Proof of Theorem \ref{thm:CruxPart}}\label{sec:ProofOfCruxResult}
Throughout this section, let $d\in \mathbb{P}$, $n>1$, ${\bf a}\in \mathbb{P}^n$, and let $A$ be a non-empty quasi-compressed $W_d$-monomial space that is not revlex-segment. Assume that $W$ is either mixed or hinged. Recall that $W$ is mixed if there is some integer $r$ satisfying $0\leq r\leq n$ such that $a_i = 1$ for all $i\in [r]$, $\lambda_i = 1$ for all $i\in [n]\backslash [r]$, and $a_{r+1}\leq \dots \leq a_n$, while $W$ is hinged if $a_1\leq \dots \leq a_n$ and $\lambda_i = 1$ for all $i\neq 1$. The purpose of this section is to prove Theorem \ref{thm:CruxPart}, i.e. we will prove the existence of some $W_d$-monomial space $A^{\prime}$ satisfying $|A^{\prime}| = |A|$, $\|A^{\prime}\|_{({\bf a}, \boldsymbol{\lambda})} < \|A\|_{({\bf a}, \boldsymbol{\lambda})}$, and $|\partial(A^{\prime})| \leq |\partial(A)|$.

Note that such a $W_d$-monomial space $A^{\prime}$ does not necessarily exist if $W$ is neither mixed nor hinged. Furthermore, already in the case when $W$ is mixed, Theorem \ref{thm:CruxPart} would simultaneously imply both the Clements-Lindstr\"{o}m theorem~\cite{ClementsLindstrom1969} and the Frankl-F\"{u}redi-Kalai theorem~\cite{FFK1988}, so any proof of Theorem \ref{thm:CruxPart} must necessarily contain the ``difficult'' parts of any proofs of these two classic theorems. A quick look at the existing proofs of these two theorems would convince the reader that a long list of lemmas is expected, thus it should not come as a surprise that our proof of Theorem \ref{thm:CruxPart} is indeed quite long and is accompanied by eight lemmas.

We first introduce some auxillary notation. For each $x\in X_{\boldsymbol{\lambda}}$, let $\tau(x)$ be the unique integer in $[n]$ such that $x\in X_{\tau(x)}$, and let $\hat{\tau}(x)$ be the unique integer in $[\lambda_{\tau(x)}]$ such that $x = x_{\tau(x), \hat{\tau}(x)}$. If $X^{\prime} \subseteq X_{\boldsymbol{\lambda}}$, then let $\max_{\leq_{r\ell}} X^{\prime}$ be the largest element of $X^{\prime}$ in the revlex order, and define $\min_{\leq_{r\ell}} X^{\prime}$ analogously. Given a monomial $m$ in $W$, define $\sigma(m) := \{\tau(x): x\in \supp(m)\} \subseteq [n]$. Whenever we say $m$ is the largest (or smallest) monomial satisfying certain conditions, it is always with respect to the revlex order. If $\deg(m) = d$ and $m$ is not the smallest monomial in $W_d$, then there is some monomial $m^{\dag}$ in $W_d$ such that $m$ {\it covers} $m^{\dag}$ in the revlex order, i.e. there does not exist any $m^{\prime} \in \{W_d\}$ satisfying $m>_{r\ell} m^{\prime} >_{r\ell} m^{\dag}$. Such a monomial $m^{\dag}$ is uniquely determined by $m$, and we reserve the superscript `$\dag$' for this cover relation.

Let $p$ be the smallest monomial in $W_d$ satisfying $p\not\in A$ and $p^{\dag} \in A$. For each $x\in X_{\boldsymbol{\lambda}}$, let $\rho(x)$ be the largest divisor of $p$ whose support is contained in $\{x^{\prime} \in X_{\boldsymbol{\lambda}}: x^{\prime} >_{r\ell} x\}$. Let $y_0 := \max_{\leq_{r\ell}}(\supp(p))$, and inductively define $y_i = y_{i-1}^{\dag}$ for $i\geq 1$. Let $j$ be the unique integer in $[d]$ such that $y_i^{a_{\tau(y_i)}}$ divides $p$ for every $i\in [j-1]$ (if any), and $y_j^{a_{\tau(y_j)}}$ does not divide $p$. Let $z := \max_{\leq_{r\ell}}\big\{x\in X_{\boldsymbol{\lambda}}: x\leq_{r\ell} y_j, \frac{p}{\rho(x)}\cdot x \neq 0\big\}$, let $p_0 := \rho(z)$, and observe that $z$ is well-defined since $p$ is not the smallest monomial in $W_d$. In fact, the definition of the revlex order tells us that $p^{\dag}$ is the largest monomial in $W_d$ divisible by $\frac{p}{p_0}\cdot z$. This means we can factorize $p^{\dag} = \frac{p}{p_0}\cdot z\cdot \widetilde{p}_0$, where $\widetilde{p}_0$ is the largest monomial of $W$ of degree $\deg(p_0)-1$ such that $\frac{p}{p_0}\cdot z\cdot \widetilde{p}_0 \neq 0$. For any variables $y, y^{\prime} \in X_{\boldsymbol{\lambda}}$ that divide $\widetilde{p}_0$, the maximality of $\widetilde{p}_0$ tells us that $\tau(y) = \tau(y^{\prime})$ if and only if $y = y^{\prime}$, thus $\supp(\widetilde{p}_0) = \{x_{i,1}: i\in \sigma(\widetilde{p}_0)\}$.

Suppose $m$ is a monomial in $\{W_d\}\backslash \{A\}$ such that $m\geq_{r\ell} p$. If $\partial(\{m\}) \subseteq \{\partial(A)\}$, then by letting $A^{\prime} := \Spank((\{A\}\backslash \{p^{\dag}\}) \cup\{m\})$, we get $|\partial(A)| \geq |\partial(A^{\prime})|$ and we are done. So for each such $m$, assume henceforth that $\partial(\{m\}) \not\subseteq \{\partial(A)\}$, and define the (non-empty) set $\Gamma_m := \{x\in \supp(m): \frac{m}{x} \not\in \{\partial(A)\}\}$.

Next, let $p_{\min}$ be the smallest monomial in $A$, and let $q$ be any monomial in $W_d$ satisfying $p_{\min} \leq_{r\ell} q \leq_{r\ell} p^{\dag}$. The minimality of $p$ implies $q\in A$, so since $A$ is quasi-compressed and $n>1$, the fact that $p\not\in A$ forces $p_{X_t} \neq q_{X_t}$ for all $t\in [n]$. In particular, when $q = p^{\dag}$, the factorization $p^{\dag} = \frac{p}{p_0}\cdot z\cdot \widetilde{p}_0$ implies $\sigma(p_0) \cup \sigma(z) \cup \sigma(\widetilde{p}_0) = [n]$.

\begin{lemma}\label{lemma:Intersection=>at>1}
$\{i\in [n]: a_i=1\} \subseteq \sigma(p_0) \cup \sigma(z)$.
\end{lemma}

\begin{proof}
Suppose not, then since $\sigma(p_0) \cup \sigma(z) \cup \sigma(\widetilde{p}_0) = [n]$, there exists some $t\in [n]$ such that $a_t=1$ and $t\in \sigma(\widetilde{p}_0) \cap \big([n]\backslash(\sigma(p_0) \cup \sigma(z))\big)$. Note that $\supp(\widetilde{p}_0) = \{x_{i,1}: i\in \sigma(\widetilde{p}_0)\}$ implies $x_{t,1}$ divides $\widetilde{p}_0$, so the factorization $p^{\dag} = \frac{p}{p_0}\cdot z\cdot \widetilde{p}_0$ gives us $\frac{p}{p_0}\cdot x_{t,1} \neq 0$. Choose some $y\in \Gamma_p$. Since $t\not\in \sigma(p_0)$, we thus get $\frac{p}{y}\cdot x_{t,1} \neq 0$, hence $\frac{p}{y}\cdot x_{t,1} \not\in A$ by the definition of $\Gamma_p$. Next, write $m:= \frac{p}{y}\cdot x_{t,1}$. We check that $\min_{\leq_{r\ell}}\{y,z\}$ is the smallest variable whose exponents in $p^{\dag} = \frac{p}{p_0}\cdot z\cdot \widetilde{p}_0$ and $m$ are different, so the definition of the revlex order yields $p^{\dag} <_{r\ell} m$. Now, $a_t=1$ implies $m_{X_t} = (p^{\dag})_{X_t} = x_{t,1}$, so since $m\not\in A$, we get $\mathcal{C}_t(A) \neq A$, which contradicts the assumption that $A$ is quasi-compressed.
\end{proof}

\begin{lemma}\label{lemma:switchingVariables}
Let $m, m^{\prime}$ be monomials in $W_d$ such that $m^{\prime} = \frac{m}{x}\cdot x^{\prime}$ for some $x\in \supp(m)$ and $x^{\prime} \in X_{\boldsymbol{\lambda}}$ satisfying $x^{\prime} \geq_{r\ell} x$. If $n>2$ and $m\in A$, then $m^{\prime} \in A$.
\end{lemma}

\begin{proof}
If $n>2$, then we can find some $t\in [n]$, distinct from $\tau(x)$ and $\tau(x^{\prime})$, such that $m_{X_t} = m^{\prime}_{X_t}$. Since $x^{\prime} \geq_{r\ell} x$ implies $m^{\prime}\geq_{r\ell} m$, it then follows from the assumption $\mathcal{C}_t(A) = A$ that $m\in A$ implies $m^{\prime} \in A$.
\end{proof}

\begin{lemma}\label{lemma:pmin-largestVar}
If $n>2$ and $a_n\neq 1$, then $\frac{p}{p_0}\cdot z\cdot x_{n,1}$ divides every $q\in \{W_d\}$ satisfying $p_{\min} \leq_{r\ell} q \leq_{r\ell} p^{\dag}$.
\end{lemma}

\begin{proof}
Let $p_0^{\prime}$ be the smallest monomial in $W$ of degree $\deg(p_0)-1$, whose support is contained in $\rho(z)$, such that $\frac{p}{p_0}\cdot z \cdot p_0^{\prime} \neq 0$. Clearly such a monomial $p_0^{\prime}$ exists, since $\widetilde{p}_0$ is the largest monomial satisfying the same properties. We first show that $\frac{p}{p_0}\cdot z \cdot p_0^{\prime} \not\in A$. Note that $\frac{p}{p_0}\cdot z \cdot p_0^{\prime} \leq_{r\ell} \frac{p}{p_0}\cdot z \cdot \widetilde{p}_0 = p^{\dag} <_{r\ell} p = \frac{p}{p_0}\cdot p_0$, hence $z\cdot p_0^{\prime} <_{r\ell} p_0$. Let $v_1 \leq_{r\ell} v_2 \leq_{r\ell} \dots \leq_{r\ell} v_{\deg(p_0)} = y_0$ be the $\deg(p_0)$ uniquely determined variables such that $p_0 = v_1v_2\cdots v_{\deg(p_0)}$. Similarly, let $z = v_1^{\prime} \leq_{r\ell} v_2^{\prime} \leq_{r\ell} \dots \leq_{r\ell} v_{\deg(p_0)}^{\prime}$ be the $\deg(p_0)$ uniquely determined variables such that $z\cdot p_0^{\prime} = v_1^{\prime}v_2^{\prime}\cdots v_{\deg(p_0)}^{\prime}$. The minimality of $p_0^{\prime}$ implies $v_i^{\prime} \leq v_i$ for each $i\in [\deg(p_0)]$, so since $p\not\in A$, the repeated use of Lemma \ref{lemma:switchingVariables} yields $\frac{p}{p_0}\cdot z \cdot p_0^{\prime} \not\in A$ as claimed.

Consequently, the minimality of $p_{\min}$ implies $p_{\min} >_{r\ell} \frac{p}{p_0}\cdot z \cdot p_0^{\prime}$, so in particular, $\frac{p}{p_0}\cdot z$ divides every monomial $q$ in $W_d$ satisfying $p_{\min} \leq_{r\ell} q \leq_{r\ell} p^{\dag}$, and we can factorize each such $q$ by $q = \frac{p}{p_0}\cdot z \cdot \widetilde{q}$, where $\widetilde{q}$ is a monomial in $W$ of degree $\deg(p_0)-1$ satisfying $\supp(\widetilde{q}) \subseteq \{x\in X_{\boldsymbol{\lambda}}: x >_{r\ell} z\}$.

Suppose we can choose some $q$ that is not divisible by $x_{n,1}$. Since $a_n\neq 1$ implies $\lambda_n=1$, we get $q_{X_n} = 1$, thus $p_{X_n} \neq 1$, i.e. $x_{n,1}$ divides $p$. Let $b = \deg(p_{X_n})$ (i.e. $p_{X_n} = x_{n,1}^b = y_0^b$), and note that $b\leq \deg(p_0)$. Let $q^{\prime}$ be the largest degree $b$ divisor of $q$. If $b<\deg(p_0)$, then $q^{\prime}$ divides $\widetilde{q}$, hence $\frac{p}{p_0}\cdot z \cdot \frac{\widetilde{q}}{q^{\prime}} \cdot x_{n,1}^b$ is a monomial in $W_d$ satisfying $p_{\min} \leq_{r\ell} \frac{p}{p_0}\cdot z \cdot \frac{\widetilde{q}}{q^{\prime}} \cdot x_{n,1}^b <_{r\ell} p$, so $\frac{p}{p_0}\cdot z \cdot \frac{\widetilde{q}}{q^{\prime}} \cdot x_{n,1}^b \in A$, which gives the contradiction $(\frac{p}{p_0}\cdot z \cdot \frac{\widetilde{q}}{q^{\prime}} \cdot x_{n,1}^b)_{X_n} = p_{X_n} = x_{n,1}^b$. Therefore, the equality $b = \deg(p_0)$ must hold, i.e. $p_0 = x_{n,1}^b$, which forces $\widetilde{p}_0 = x_{n,1}^{b-1}$. However, the fact that $\sigma(p_0) \cup \sigma(z) \cup \sigma(\widetilde{p}_0) = [n]$ would contradict the assumption $n>2$.
\end{proof}

\begin{lemma}\label{lemma:Wmixed=>p=p0}
If $p\neq p_0$, then $W$ is hinged, $a_1\neq 1$, and $\tau(z) = 1$.
\end{lemma}

\begin{proof}
Suppose there exists some $y\in \supp(\frac{p}{p_0})$. The definition of $p_0$ tells us that $z$ is strictly smaller than every variable in $\supp(p_0)$ or $\supp(\widetilde{p}_0)$, so $\sigma(p_0) \cup \sigma(z) \cup \sigma(\widetilde{p}_0) = [n]$ implies $y\leq_{r\ell} z\leq_{r\ell} x_{1,1}$. If $a_1 = 1$, then whether $W$ is mixed or hinged, it follows from $y\leq_{r\ell} x_{1,1}$ that $a_{\tau(y)} = 1$, which forces the contradiction $p_{X_{\tau(y)}} = (p^{\dag})_{X_{\tau(y)}} = y$. Thus $a_1 \neq 1$, therefore $W$ is hinged and $\tau(z) = 1$.
\end{proof}

\begin{lemma}\label{lemma:largestvardividingp}
If $a_n\neq 1$ and $x_{n,1}^{a_n}$ divides $p$, then $p_0 = x_{n,1}^{a_n}$, $z = x_{n-1,1}$, and $n=2$.
\end{lemma}

\begin{proof}
Clearly if $x_{n,1}^{a_n}$ divides $p$, then $x_{n,1}^{a_n}$ divides $p_0$. Furthermore, if $\deg(p_0) > a_n$, then $\deg(\widetilde{p}_0) \geq a_n$, which yields the contradiction $p_{X_n} = (p^{\dag})_{X_n} = x_{n,1}^{a_n}$, where the last equality follows from the fact that $a_n\neq 1$ implies $\lambda_n = 1$. Thus $p_0 = x_{n,1}^{a_n}$. Next, we show that $z = x_{n-1,1}$. Suppose instead $z<_{r\ell} x_{n-1,1}$, then by the maximality of $z$, we get $\deg(p_{X_{n-1}}) = a_{n-1}$, so the definition of $j$ forces $\lambda_{n-1} > 1$. Also, since $n-1\not\in \sigma(p_0)$ implies $\frac{p}{p_0} \neq 1$, Lemma \ref{lemma:Wmixed=>p=p0} yields $W$ is hinged, thus $n=2$ in this case, and by the maximality of $z$, the largest variable in $\supp(p_{X_{n-1}}) = \supp(p_{X_1})$ must divide $p_0$, which is a contradiction. Therefore $z = x_{n-1,1}$, and so $\widetilde{p}_0 = x_{n,1}^{a_n-1}$. Finally, since $\sigma(p_0) \cup \sigma(z) \cup \sigma(\widetilde{p}_0) = [n]$, we get $n=2$.
\end{proof}

\begin{lemma}\label{lemma:possibleshadows}
If $y\in \Gamma_p$, then either $\tau(y) = \tau(z)$ or $a_{\tau(y)} > 1$.
\end{lemma}

\begin{proof}
Assume there is some $y\in \Gamma_p$ such that $\tau(y) \neq \tau(z)$ and $a_{\tau(y)} = 1$. Note that $\frac{p}{y}\cdot x_{\tau(y),1}$ is a monomial in $\{W_d\}\backslash \{A\}$ that satisfies $\frac{p}{y}\cdot x_{\tau(y),1} \geq_{r\ell} p^{\dag}$. This implies $\tau(y) \not\in \sigma(\widetilde{p}_0)$, since otherwise $a_{\tau(y)} = 1$ forces $\big(\frac{p}{y}\cdot x_{\tau(y),1}\big)_{X_{\tau(y)}} = (p^{\dag})_{X_{\tau(y)}} = x_{\tau(y),1}$, which contradicts the assumption that $A$ is quasi-compressed.

Next, we show that $\{i\in [n]: a_i=1\} \neq [n]$. Suppose not, then by Lemma \ref{lemma:Intersection=>at>1}, either (i): $\sigma(p_0) = [n]$, $\tau(z) \in \sigma(p_0)$, and $\deg(p_0) = n$; or (ii): $\sigma(p_0) = [n]\backslash \sigma(z)$, $\tau(z) \not\in \sigma(p_0)$, and $\deg(p_0) = n-1$. Note that $p = p_0$ for both cases. In case (i), the definition of $\widetilde{p}_0$ yields $\deg(\widetilde{p}_0) = n-1$ and $\tau(z) \not\in \sigma(\widetilde{p}_0)$, so it follows from $\tau(y) \not\in \sigma(\widetilde{p}_0)$ that $\tau(y) = \tau(z)$, which is a contradiction. In case (ii), note that if $\hat{\tau}(z)=1$, then $z = x_{1,1}$ and $p = x_{2,1}\dots x_{n,1}$, which contradict the assumptions that $p_{X_n} \neq (p^{\dag})_{X_n}$ and $\partial(\{p\}) \not\subseteq \{\partial(A)\}$ when $n>2$ and $n=2$ respectively. Hence $\hat{\tau}(z)>1$, and $\frac{p}{y}\cdot x_{\tau(z),1}$ is a monomial in $\{W_d\}\backslash \{A\}$ that satisfies $\frac{p}{y}\cdot x_{\tau(z),1} >_{r\ell} p^{\dag}$. Since $\tau(y) \not\in \sigma(\widetilde{p}_0)$, we then get the contradiction $\big(\frac{p}{y}\cdot x_{\tau(z),1}\big)_{X_{\tau(y)}} = (p^{\dag})_{X_{\tau(y)}} = 1$.

Thus, $\{i\in [n]: a_i=1\} \neq [n]$ as claimed, and in particular, whether $W$ is mixed or hinged, we must have $a_n \neq 1$ and $\lambda_n = 1$. It is easy to see that $x_{n,1}^{a_n}$ divides $p$, since otherwise $\frac{p}{y}\cdot x_{n,1}$ is a monomial in $\{W_d\}\backslash \{A\}$ that satisfies $\frac{p}{y}\cdot x_{n,1} \geq_{r\ell} p^{\dag}$, which gives the contradiction $\big(\frac{p}{y}\cdot x_{n,1}\big)_{X_{\tau(y)}} = (p^{\dag})_{X_{\tau(y)}} = 1$. Consequently, Lemma \ref{lemma:largestvardividingp} yields $p_0 = x_{n,1}^{a_n}$, $z = x_{n-1,1}$, and $n=2$, so the assumption $a_{\tau(y)} = 1$ forces $a_{n-1} = a_1 = 1$, $\tau(y) = 1$, and $y\in \supp(\frac{p}{p_0})$. Yet, the existence of $y$ implies $p\neq p_0$, which contradicts Lemma \ref{lemma:Wmixed=>p=p0}.
\end{proof}

\begin{lemma}\label{lemma:shadowcontribution=1}
Suppose $\Gamma_p$ contains a unique element $z^{\prime}$, and suppose $\{i\in [n]: a_i=1\} = [n]$. If $m$ is a monomial in $W_d$ that is divisible by $\frac{p}{z^{\prime}}$ and satisfies $m\geq_{r\ell} p$, then $|\partial(\{m\}) - \partial(\{A\})| = 1$.
\end{lemma}

\begin{proof}
From the assumption $\{i\in [n]: a_i=1\} = [n]$, Lemma \ref{lemma:possibleshadows} tells us $\tau(z^{\prime}) = \tau(z)$, while Lemma \ref{lemma:Wmixed=>p=p0} gives us $p=p_0$, hence $\tau(z^{\prime}) \in \sigma(p_0)$, and Lemma \ref{lemma:Intersection=>at>1} yields $\sigma(p) = \sigma(p_0) = [n]$, which forces $d=n$. Write $p^{\prime} = \frac{p}{z^{\prime}}$ and fix some monomial $m = p^{\prime}\cdot y^{\prime}$ in $W_d$, where $y^{\prime} \in X_{\boldsymbol{\lambda}}$ satisfies $y^{\prime} \geq_{r\ell} z^{\prime}$. Note that $d=n$ implies $\tau(y^{\prime}) = \tau(z^{\prime})$. The case $\supp(m) = \{y^{\prime}\}$ is trivial, so assume not, and fix some $y\in \supp(m)$ such that $y\neq y^{\prime}$. In particular, $y\neq y^{\prime}$ implies $\tau(y) \neq \tau(y^{\prime})$. Clearly $y\in \supp(p^{\prime})$, so since $a_{\tau(z^{\prime})} = 1$ means $z^{\prime}$ does not divide $p^{\prime}$, we get $y\neq z^{\prime}$, hence the uniqueness of $z^{\prime}$ yields $\frac{p}{y} \in \partial(A)$. Consequently, the fact that $A$ is quasi-compressed implies $\widetilde{p} := \frac{p}{y}\cdot x_{\tau(y),1} \in \{A\}$. To prove this lemma, we have to show that $\frac{m}{y} \in \partial(A)$. Note that $\frac{m}{y} = \frac{p^{\prime}\cdot y^{\prime}}{y} = \frac{p}{z^{\prime}y} \cdot y^{\prime}$ (where $\frac{p}{z^{\prime}y}$ is a monomial), thus $\frac{m}{y}\cdot x_{\tau(y),1} = \frac{\widetilde{p}}{z^{\prime}} \cdot y^{\prime} \in \{W_d\}$. Now since $\widetilde{p} \in \{A\}$, $y^{\prime}\geq_{r\ell} z^{\prime}$, $\tau(y^{\prime}) = \tau(z^{\prime})$, and since $A$ is quasi-compressed, we get $\frac{m}{y}\cdot x_{\tau(y),1} \in \{A\}$, i.e. $\frac{m}{y} \in \partial(A)$.
\end{proof}

\begin{lemma}\label{lemma:pmin-pmax}
If $n>2$ and $a_n\neq 1$, then there exists a monomial $\widetilde{m} \in W_d$ not contained in $A$, such that $\widetilde{m} \geq_{r\ell} p$ and $\Gamma_{\widetilde{m}} = \{x_{n,1}\}$.
\end{lemma}

\begin{proof}
Fix some $y\in \Gamma_p$. Since $n>2$, $a_n\neq 1$ and hence $\lambda_n = 1$, it follows from Lemma \ref{lemma:largestvardividingp} that $x_{n,1}^{a_n}$ does not divide $p$, thus $\frac{p}{y}\cdot x_{n,1}$ and $m_0 := \frac{p}{y_0}\cdot x_{n,1}$ are monomials in $W_d$. The definition of $\Gamma_p$ implies $\frac{p}{y}\cdot x_{n,1} \not\in A$, so Lemma \ref{lemma:switchingVariables} yields $m_0 \not\in A$, which in particular means $x_{n,1} \in \Gamma_{m_0}$. Write $b_0 := \deg((m_0)_{X_n})$, $b := \min\{a_n, \deg(p_0)-1\}$, and note that $(p^{\dag})_{X_n} = x_{n,1}^b$ by construction, so $(m_0)_{X_n} \neq (p^{\dag})_{X_n}$ implies $b_0 \neq b$.

If $y_0 = x_{n,1}$, then $m_0 = p$, so $x_{n,1}^{b_0}$ divides $p_0$, which implies $b_0\leq \min\{a_n, \deg(p_0)\}$. Also, we either have $\deg(p_0)-1 \geq a_n$, or $b_0 < \deg(p_0)$, since otherwise we would get $p_0 = x_{n,1}^{b_0}$, which forces $\widetilde{p}_0 = x_{n,1}^{b_0-1}$, and the fact that $\sigma(p_0) \cup \sigma(z) \cup \sigma(\widetilde{p}_0) = [n]$ would then contradict the assumption $n>2$. Thus, in either case, it follows from $b_0 \neq b$ that $1\leq b_0 < b$. If instead $y_0\neq x_{n,1}$, then since Lemma \ref{lemma:pmin-largestVar} yields $b \neq 0$, we get $1 = b_0 < b$. Consequently, whether $y_0 = x_{n,1}$ or $y_0 \neq x_{n,1}$, we always have $1\leq b_0 < b$.

Now if $\Gamma_{m_0} = \{x_{n,1}\}$, then $\widetilde{m} = m_0$ is our desired monomial. If not, then let $v_0$ be the largest variable in $\Gamma_{m_0}\backslash \{x_{n,1}\}$ and define $m_1 := \frac{m_0}{v_0}\cdot x_{n,1}$, which is a monomial in $W_d$ with $(m_1)_{X_n} = x_{n,1}^{b_0+1}$. Since $m_0 \not\in A$, Lemma \ref{lemma:switchingVariables} tells us that $m_1\not\in A$, so it follows from $m_1 >_{r\ell} m_0 >_{r\ell} p^{\dag}$ that $(m_1)_{X_n} \neq (p^{\dag})_{X_n}$, i.e. $b_0+1 < b$. In general, if $\Gamma_{m_i} \neq \{x_{n,1}\}$ and $b_0+i<b$ for some $i$, then we let $v_i$ be the largest variable in $\Gamma_{m_i}\backslash \{x_{n,1}\}$ (which is well-defined since $\Gamma_{m_i}$ is non-empty) and define $m_{i+1} := \frac{m_i}{v_i}\cdot x_{n,1}$, which is a monomial in $W_d$ with $(m_{i+1})_{X_n} = x_{n,1}^{b_0+i+1}$. The same argument as before gives $m_{i+1}\not\in A$ and $b_0+i+1<b$. Repeat this process until we get some integer $i^{\prime} < b-b_0$ such that $\Gamma_{m_{i^{\prime}}} = \{x_{n,1}\}$. Then $\widetilde{m} = m_{i^{\prime}}$ is our desired monomial.
\end{proof}

\begin{paragraph}{\it Proof of Theorem \ref{thm:CruxPart}.}
Consider three cases (i): $\{i\in [n]: a_i = 1\} = [n]$; (ii) $\{i\in [n]: a_i = 1\} \neq [n]$, $n>2$; and (iii): $\{i\in [n]: a_i = 1\} \neq [n]$, $n=2$.

{\it Case (i):} Suppose $\{i\in [n]: a_i = 1\} = [n]$, then Lemma \ref{lemma:possibleshadows} tells us $\Gamma_p$ contains a unique element $z^{\prime}$, which satisfies $\tau(z^{\prime}) = \tau(z)$. Lemma \ref{lemma:Wmixed=>p=p0} implies $p=p_0$, thus $\tau(z^{\prime}) \in \sigma(p_0)$, and Lemma \ref{lemma:Intersection=>at>1} gives $\sigma(p) = \sigma(p_0) = [n]$, which implies $d=n$. Let $\widetilde{y} := \max_{\leq_{r\ell}}(\supp(p_{\min}))$ and $\widetilde{q} := \frac{p_{\min}}{\widetilde{y}} \in \partial(\{A\})$. Note that $z = \min_{\leq_{r\ell}}(\supp(p^{\dag}))$, so since $p_{\min} \leq_{r\ell} p^{\dag}$, we get $\min_{\leq_{r\ell}}(\supp(\widetilde{q})) \leq_{r\ell} z$, thus $\widetilde{q} \neq \frac{p}{z^{\prime}}$. Define the sets of monomials $\mathcal{A}^{\prime} := \big\{m\in \{W_d\}: m\geq_{r\ell} p, \tfrac{p}{z^{\prime}}\text{ divides }m\big\}$ and $\mathcal{A}^{\prime\prime} := \big\{m\in \{A\}: \widetilde{q} \text{ divides }m\big\}$. For each variable $x\in X_{\boldsymbol{\lambda}}$, recall that $\hat{\tau}(x)$ is the unique integer in $[\lambda_{\tau(x)}]$ satisfying $x = x_{\tau(x), \hat{\tau}(x)}$. Since $d=n$ and $A$ is quasi-compressed, we get $|\mathcal{A}^{\prime}| = \hat{\tau}(z^{\prime})$ and $|\mathcal{A}^{\prime\prime}| = \hat{\tau}(\widetilde{y})$.

Next, we show $|\mathcal{A}^{\prime}| \geq |\mathcal{A}^{\prime\prime}|$. Suppose instead $\hat{\tau}(z^{\prime}) < \hat{\tau}(\widetilde{y})$. Since $p = p_0$ and $a_{\tau(z^{\prime})} = 1$, the maximality of $z$ gives $\hat{\tau}(z^{\prime}) = \hat{\tau}(z)-1$, so $\hat{\tau}(z) \leq \hat{\tau}(\widetilde{y})$. If $\tau(\widetilde{y}) \neq \tau(z)$, then $d=n$ implies $\min_{\leq_{r\ell}}(\supp(p_{\min})) \leq_{r\ell} z$, with equality holding if and only if $\sigma(p_{\min}) = [n]$ and $\hat{\tau}(x) = \hat{\tau}(z)$ for all $x\in \supp(p_{\min})$. If $\tau(\widetilde{y}) = \tau(z)$, then $\widetilde{y} \leq_{r\ell} z$, and $d=n\geq 2$ implies $\min_{\leq_{r\ell}}(\supp(p_{\min})) <_{r\ell} \max_{\leq_{r\ell}}(\supp(p_{\min})) \leq_{r\ell} z$. In either case, $p_{\min} \leq_{r\ell} \frac{p}{z^{\prime}}\cdot z <_{r\ell} p$, hence $\frac{p}{z^{\prime}}\cdot z\in A$, which yields the contradiction $(\frac{p}{z^{\prime}}\cdot z)_{X_t} = p_{X_t}$ for all $t\in [n]\backslash \{\tau(z)\}$.

Now let $\mathcal{A} := (\{A\} \cup \mathcal{A}^{\prime})\backslash \mathcal{A}^{\prime\prime}$, and note that $|\mathcal{A}| = |A| + |\mathcal{A}^{\prime}| - |\mathcal{A}^{\prime\prime}| \geq |A|$. Let $A^{\prime}$ be the $W_d$-monomial space spanned by the $|A|$ largest monomials in $\mathcal{A}$, and observe that $\|A^{\prime}\|_{({\bf a}, \boldsymbol{\lambda})} < \|A\|_{({\bf a}, \boldsymbol{\lambda})}$ by construction. Finally, Lemma \ref{lemma:shadowcontribution=1} tells us that $\partial(\mathcal{A}^{\prime}) - \partial(\{A\}) = \big\{\tfrac{p}{z^{\prime}}\big\}$, so since $\widetilde{q} \neq \frac{p}{z^{\prime}}$, we get $|\partial(A^{\prime})| \leq |\partial(A)| + 1 - 1 = |\partial(A)|$, and we are done with this case.

{\it Case (ii):} Suppose instead $\{i\in [n]: a_i = 1\} \neq [n]$ and $n>2$. Note that $a_n\neq 1$, so Lemma \ref{lemma:pmin-pmax} tells us there exists a monomial $\widetilde{m}$ in $W_d$ not contained in $A$, such that $\widetilde{m} \geq_{r\ell} p$ and $\Gamma_{\widetilde{m}} = \{x_{n,1}\}$. This means $x_{n,1}$ divides $\widetilde{m}$, and $\partial(\{A\}\cup \{\widetilde{m}\}) = \big(\partial(\{A\})\big) \cup \big\{\frac{\widetilde{m}}{x_{n,1}}\big\}$. Lemma \ref{lemma:pmin-largestVar} says $x_{n,1}$ divides $p_{\min}$, so the minimality of $p_{\min}$ yields $\frac{p_{\min}}{x_{n,1}} \not\in \partial\big(\{A\}\backslash \{p_{\min}\}\big)$. Furthermore, Lemma \ref{lemma:pmin-largestVar} also tells us $\frac{p}{p_0}\cdot z$ divides $p_{\min}$, so it follows from $\widetilde{m} \geq_{r\ell} p = \frac{p}{p_0}\cdot p_0$ that the exponents of $z$ in $p_{\min}$ and $\widetilde{m}$ are different. Consequently, $\frac{p_{\min}}{x_{n,1}} \neq \frac{\widetilde{m}}{x_{n,1}}$, thus $\partial\big((\{A\}\backslash \{p_{\min}\}) \cup \{\widetilde{m}\}\big) \subseteq \big((\partial(\{A\})\backslash \big\{\tfrac{p_{\min}}{x_{n,1}}\big\}\big) \cup \big\{\tfrac{\widetilde{m}}{x_{n,1}}\big\}$. Therefore, if $A^{\prime} := \Spank\big((\{A\}\backslash \{p_{\min}\}) \cup \{\widetilde{m}\}\big)$, then we have $|A^{\prime}| = |A|$, $\|A^{\prime}\|_{({\bf a}, \boldsymbol{\lambda})} < \|A\|_{({\bf a}, \boldsymbol{\lambda})}$, and $|\partial(A^{\prime})| \leq |\partial(A)|$.

{\it Case (iii):} Finally, consider the case $\{i\in [n]: a_i = 1\} \neq [n]$ and $n=2$. Equivalently, we have the conditions $a_1\leq a_2$, $a_2>1$, and $\lambda_2 = 1$. It is an easy consequence of Macaulay's theorem that $Q_{(a_1)} = \langle X_1\rangle^{a_1+1}$ is a Macaulay-Lex ideal of $\fieldk[X_1]$ with respect to the linear order $x_{1,1} < \dots < x_{1,\lambda_1}$ on its variables (cf. Lemma \ref{lemma:truncated-n=1-case}), hence the colored quotient ring $\fieldk[X_{\boldsymbol{\lambda}}]/Q_{(a_1,\infty)}$ is Macaulay-Lex by Theorem \ref{thm:MacLexProperties}\ref{thm-part:extendNewVarMacLex}. If $d\leq a_2$, then $W_d = (\fieldk[X_{\boldsymbol{\lambda}}]/Q_{(a_1,\infty)})_d$, which implies $\Revlex_{W_d}(A) = \Revlex_{(\fieldk[X_{\boldsymbol{\lambda}}]/Q_{(a_1,\infty)})_d}(A)$, so Corollary \ref{cor:coloredMacLexChar} yields $|\partial(\Revlex_{W_d}(A))| \leq |\partial(A)|$, and we are done by letting $A^{\prime} = \Revlex_{W_d}(A)$.

Assuming $a_2 < d\leq a_1 + a_2$, it follows from $a_1 \leq a_2$ that $x_{2,1}$ divides every monomial in $W_d$, and we note that $p_{\min}$ is the largest monomial in $W_d$ divisible by $\frac{p_{\min}}{x_{2,1}}$. Consequently, if we can find a monomial $\widetilde{m} \in \{W_d\}\backslash \{A\}$ such that $\widetilde{m} \geq_{r\ell} p$ and $|\Gamma_{\widetilde{m}}| = 1$, then $A^{\prime} := \Spank\big((\{A\} \cup \{\widetilde{m}\}) \backslash \{p_{\min}\}\big)$ satisfies $|A^{\prime}| = |A|$, $\|A^{\prime}\|_{({\bf a}, \boldsymbol{\lambda})} < \|A\|_{({\bf a}, \boldsymbol{\lambda})}$, and $|\partial(A^{\prime})| \leq |\partial(A)| + |\Gamma_{\widetilde{m}}| - \big|\big\{\tfrac{p_{\min}}{x_{2,1}}\big\}\big| = |\partial(A)|$, and we would be done.

We will show that such a monomial $\widetilde{m}$ exists. For each monomial $m\in \{W_d\}\backslash \{A\}$ such that $m\geq_{r\ell} p$, let $m_{\max}$ denote the largest monomial in $\{W_d\}\backslash \{A\}$ satisfying $(m_{\max})_{X_2} = m_{X_2}$. Note that $\frac{m_{\max}}{x} \cdot x_{1,1} \geq_{r\ell} m_{\max}$ for each $x\in \supp(m_{\max}) \cap X_1$, with equality holding if and only if $x = x_{1,1}$, thus the maximality of $m_{\max}$ implies $\Gamma_{m_{\max}} \cap X_1 \subseteq \{x_{1,1}\}$, i.e. $\Gamma_{m_{\max}} \subseteq \{x_{1,1}, x_{2,1}\}$. If $x_{1,1} \not\in \Gamma_{m_{\max}}$, then since $\Gamma_{m_{\max}}$ is assumed to be non-empty, we get $\Gamma_{m_{\max}} = \{x_{2,1}\}$, so we can let $\widetilde{m} = m_{\max}$ and we are done. Consequently, we assume henceforth that $x_{1,1} \in \Gamma_{m_{\max}}$ for all monomials $m\in \{W_d\}\backslash \{A\}$ satisfying $m\geq_{r\ell} p$.

Let $p^*$ denote the largest monomial in $\{W_d\}\backslash \{A\}$. Among all pairs $(q, q^{\dag})$ of consecutive monomials in $W_d$ such that $q\not\in A$ and $q^{\dag} \in A$, choose a pair for which $\deg(q_{X_1})$ is minimized. Clearly $p^* = p^*_{\max}$, and note that $x_{1,1}$ is, by assumption, contained in each of $\Gamma_{p^*}$ and $\Gamma_{q_{\max}}$. In particular, $x_{1,1} \in \Gamma_{p^*}$ implies $x_{2,1}^{a_2}$ divides $p^*$, since otherwise we would have $\frac{p^*}{x_{1,1}}\cdot x_{2,1} \in \{W_d\}\backslash \{A\}$ by the definition of $\Gamma_{p^*}$, and the fact that $\frac{p^*}{x_{1,1}}\cdot x_{2,1} >_{r\ell} p^*$ would then contradict the maximality of $p^*$.

Next, we show that $\deg(q_{X_1}) < a_1$. Suppose instead $\deg(q_{X_1}) = a_1$. For each $0\leq i \leq a_1+a_2-d$, define $m_i := \frac{p^*}{x_{2,1}^i} \cdot x_{1,1}^i$, which is a monomial in $W_d$ since the fact that $x_{2,1}^{a_2}$ divides $p^*$ implies $\deg((m_i)_{X_1}) = d-a_2+i$. Observe that $m_0 >_{r\ell} m_1 >_{r\ell} > \dots >_{r\ell} m_{a_1+a_2-d}$ are $(a_1+a_2-d+1)$ consecutive monomials in $W_d$, so since $p^* \geq_{r\ell} q_{\max}$ and since $\deg((m_{a_1+a_2-d})_{X_1}) = \deg((q_{\max})_{X_1}) = a_1$ by assumption, our choice of $(q, q^{\dag})$ (where $\deg(q_{X_1})$ is minimized) implies these $(a_1+a_2-d+1)$ consecutive monomials are all not in $A$. However, $q_{\max} >_{r\ell} p^{\dag}$, and there exists some $0\leq i^{\prime} \leq a_1+a_2-d$ such that $(m_{i^{\prime}})_{X_2} = (p^{\dag})_{X_2}$, which contradicts the assumption that $A$ is quasi-compressed. Consequently, $\deg(q_{X_1}) < a_1$ as claimed. 

This means $\deg((q_{\max})_{X_1}) < a_1$, which implies $q^{\dag} = \frac{q}{x_{2,1}}\cdot x_{1,1}$ and $(q_{\max})^{\dag} = \frac{q_{\max}}{x_{2,1}}\cdot x_{1,1}$, so $((q_{\max})^{\dag})_{X_2} = (q^{\dag})_{X_2}$. Now, $q^{\dag} \in A$, and $A$ is quasi-compressed, hence $(q_{\max})^{\dag} \in A$, thus $\frac{q_{\max}}{x_{2,1}} \in \partial(A)$ and $x_{2,1} \not\in \Gamma_{q_{\max}}$, i.e. $\Gamma_{q_{\max}} = \{x_{1,1}\}$, therefore we can let $\widetilde{m} = q_{\max}$.
\end{paragraph}

\section{Generalized colored multicomplexes}\label{sec:GenColoredMulticomplexes}
In this section, we study the $f$-vectors of generalized colored multicomplexes with arbitrarily prescribed maximum possible degrees of its variables. We begin by explaining the terminology we use. We then introduce what we call ``truncations'' of colored quotient rings, and we will construct a class of non-Macaulay-Lex truncations of colored quotient rings. Finally, we will use this construction to prove the main result (Theorem \ref{thm:revlexdoesntcharacterize}) of this section, which implies Theorem \ref{main-thm-f-vectors-colored}.

Let $X$ be a (possibly infinite) non-empty set of variables, and let $\mathcal{M}_X$ be the collection of all monomials whose supports are contained in $X$. We say $(X_1, \dots, X_n)$ is an {\it ordered partition} of $X$ if $X_1, \dots, X_n$ are pairwise disjoint non-empty subsets of $X$ whose union equals $X$. A {\it multicomplex} on $X$ is a subcollection $M \subseteq \mathcal{M}_X$ that is closed under divisibility, i.e. if $m\in M$ and $m^{\prime}$ divides $m$, then $m^{\prime} \in M$. In this paper, we always assume $1\in M$, and we do not require every $x\in X$ to be in $M$. For each $d\in \mathbb{N}$, let $\mathcal{M}_X^d$ be the collection of all monomials in $\mathcal{M}_X$ of degree $d$, and let $M^d$ be the collection of all monomials in $M$ of degree $d$. The {\it $f$-vector} of $M$ is $(f_0, f_1, \dots)$, where $f_i = |M^i|$ for each $i\in \mathbb{N}$. Note that $f_0 = 1$, corresponding to $1\in M$. If ${\bf a} = (a_1, \ldots, a_n) \in \mathbb{P}^n$, then an {\it ${\bf a}$-colored multicomplex} is a pair $(M, \pi)$, where $M$ is a multicomplex on a non-empty set $X$, and $\pi = (X_1, \dots, X_n)$ is an ordered partition of $X$, such that $\deg(m_{X_i}) \leq a_i$ for all $m\in M$ and $i\in [n]$.

Suppose $\varphi: X^{\prime} \to \overline{\mathbb{P}}$ is a map whose domain contains $X$. Let $\mathcal{M}_X(\varphi)$ be the collection of all monomials $m \in \mathcal{M}_X$ such that the exponent of every $x\in \supp(m)$ in $m$ is at most $\varphi(x)$. For each $d\in \mathbb{N}$, define $\mathcal{M}_X^d(\varphi) := \mathcal{M}_X^d \cap \mathcal{M}_X(\varphi)$. In particular, if ${\bf 1}: X \to \overline{\mathbb{P}}$ is the constant map defined by $x\mapsto 1$ for all $x\in X$, then $\mathcal{M}_X({\bf 1})$ is the collection of all squarefree monomials in $\mathcal{M}_X$. By abuse of notation, an ${\bf a}$-colored multicomplex $(M, \pi)$ is said to be in $\mathcal{M}_X(\varphi)$ if $M\subseteq \mathcal{M}_X(\varphi)$.

Recall that a {\it simplicial complex} $\Delta$ on a vertex set $V$ is a collection of subsets of $V$ such that $\{v\} \in \Delta$ for all $v\in V$, and $\Delta$ is closed under set inclusion. For our purposes, assume $\Delta$ is finite and non-empty. Each $F\in \Delta$ is called a {\it face}, and it has {\it dimension} $\dim F := |F| - 1$. The {\it dimension} of $\Delta$, denoted by $\dim \Delta$, is the maximum dimension of its faces, and the {\it $f$-vector} of $\Delta$ is $(f_0, \dots, f_{\dim \Delta})$, where each $f_i$ is the number of $i$-dimensional faces in $\Delta$. Every simplicial complex $\Delta$ with vertex set $V \subseteq X$ corresponds bijectively to a finite multicomplex $M \subseteq \mathcal{M}_X({\bf 1})$ via $\{x_{i_1}, \dots, x_{i_t}\} \in \Delta \Leftrightarrow {x_{i_1}}\cdots {x_{i_t}} \in M$, thus multicomplexes can be considered as generalizations of simplicial complexes. Consequently, an ${\bf a}$-colored complex with vertex set $V\subseteq X$ (as defined in Section \ref{sec:Intro}) can be identified with an ${\bf a}$-colored multicomplex in $\mathcal{M}_X({\bf 1})$. Note however that the $f$-vector of $\Delta$ and the $f$-vector of $M$ differ by a shift in the indexing.

Given a map $\phi: X_{\boldsymbol{\lambda}} \to \overline{\mathbb{P}}$, define the graded ring $W(\phi) = \bigoplus_{d\in \mathbb{N}} W_d(\phi) := \fieldk[X_{\boldsymbol{\lambda}}]/(Q_{{\bf a}} + T^{\phi})$, where $T^{\phi} := \langle\{x^{\phi(x)+1}: x\in X_{\boldsymbol{\lambda}}\}\rangle$ is an ideal contained in $\fieldk[X_{\boldsymbol{\lambda}}]$. (Recall that $W = \fieldk[X_{\boldsymbol{\lambda}}]/Q_{{\bf a}}$.) We say $W(\phi)$ is the {\it truncation} of $W$ induced by $\phi$. For the following proposition, recall that $\boldsymbol{\delta}_n = (0,\dots, 0,1)$.

\begin{proposition}\label{prop:counter-eg-a>d+1}
Let $n,d>1$, $\boldsymbol{\lambda} \in \mathbb{P}^n$ and ${\bf a} \in \overline{\mathbb{P}}^n$, such that $\boldsymbol{\lambda} \geq 2\cdot {\bf 1}_n + \boldsymbol{\delta}_n$, ${\bf a} \not\geq d\cdot {\bf 1}_n$, ${\bf a} \neq {\bf 1}_n$, and $|{\bf a}| > d$. Suppose $\phi: X_{\boldsymbol{\lambda}} \to \overline{\mathbb{P}}$ is a map satisfying $\sum_{x\in X_t\backslash\{x_{t,1}\}} \phi(x) \geq a_t$ for every $t\in [n]$. Then there exists some $W_d(\phi)$-monomial space $A$ such that $|\partial(A)| < |\partial(\Revlex_{W_d(\phi)}(A))|$.
\end{proposition}

\begin{proof}
We will explicitly construct such a $W_d(\phi)$-monomial space $A$. First, we fix some notation. For all monomials $m, m^{\prime}$ in $W(\phi)$, write $m\nmid m^{\prime}$ to mean $m$ does not divide $m^{\prime}$. Whenever we say $m$ (resp. $x\in X_{\boldsymbol{\lambda}}$) is the largest or smallest monomial (resp. variable) satisfying certain conditions, it is always with respect to the revlex order (induced by the linear order on $X_{\boldsymbol{\lambda}}$ given in Section \ref{sec:Intro}). Given any $s\in \mathbb{N}$, let $\overline{s}$ denote the unique integer in $[n]$ such that $s \equiv \overline{s} \pmod{n}$. For each $x \in X_{\boldsymbol{\lambda}}$, write $x^{\dag}$ to mean the largest variable in $X_{\boldsymbol{\lambda}}$ that is smaller than $x$. For each $t\in [n]$, define $\widehat{X}_t := \big\{m\in \mathcal{M}_{X_t}^{a_t}(\phi): x_{t,1} \nmid m\big\}$, and let $\widehat{m}_t$ be the largest monomial in $\widehat{X}_t$, which is well-defined as $\sum_{x\in X_t\backslash\{x_{t,1}\}} \phi(x) \geq a_t$ implies $\widehat{X}_t$ is non-empty. Since ${\bf a} \not\geq d\cdot {\bf 1}_n$ and ${\bf a} \neq {\bf 1}_n$, there exists some $s \in [n]$ such that $a_s \leq d-1$ and $a_{\overline{s-1}} \geq 2$. If $s>1$, then define
\begin{equation*}
\widehat{Y} := \Big\{m\in \mathcal{M}_{X_{s-1}}^{a_{s-1}-2}(\phi) : x_{s-1,1} \nmid m \text{ and }x_{s-1,2}^{\min\{\phi(x_{s-1,2}), a_{s-1}\}} \nmid m \Big\},
\end{equation*}
while if $s=1$, then define
\begin{equation*}
\widehat{Y} := \Big\{m\in \mathcal{M}_{X_n}^{a_n-2}(\phi) : x_{n,1} \nmid m \text{ and }y^{\min\{\phi(y), a_n\}} \nmid m \text{ for each }y\in \{x_{n,2}, x_{n,3}\}\Big\}.
\end{equation*}
In either case, $\widehat{Y}$ is non-empty, since $\sum_{x\in X_{\overline{s-1}}\backslash\{x_{\overline{s-1},1}\}} \phi(x) \geq a_{\overline{s-1}}$. In particular, if $a_{\overline{s-1}} = 2$, then $\widehat{Y} = \{1\}$. Let $\widehat{p}$ be the largest monomial in $\widehat{Y}$, let $\Gamma := [n]\backslash \{s, \overline{s-1}\}$, and let $\widetilde{m} := \widehat{p} \cdot \prod_{t\in \Gamma} \widehat{m}_t$, which by construction is a monomial in $W(\phi)$ and has degree $|{\bf a}| - a_s - 2 \geq d-1-a_s$. Let $m^{\prime}$ be the smallest monomial of degree $d-1-a_s$ in $W(\phi)$ that divides $\widetilde{m}$, let $q := m^{\prime}\widehat{m}_s$, and let $\widetilde{q} := \frac{q(x_{s,2}^{\dag})}{x_{s,2}}$. In particular, $\boldsymbol{\lambda} \geq 2\cdot {\bf 1}_n + \boldsymbol{\delta}_n$ implies $x_{s,2}^{\dag} \in X_{\overline{s-1}}$, while $x_{s,2}$ divides $\widehat{m}_s$ by construction, so $q$ and $\widetilde{q}$ are monomials in $W_{d-1}(\phi)$, with $\widetilde{q} <_{r\ell} q$.

Next, let $\widetilde{x}$ be the smallest variable in $\{x\in X_{\boldsymbol{\lambda}}: x >_{r\ell} x_{s,2}\}$ such that $\widetilde{q}\widetilde{x} \in \{W_d(\phi)\}$. Define the sets $\widetilde{\mathcal{A}} := \big\{m\in \{W_d(\phi)\}: m \geq_{r\ell} \widetilde{q}\widetilde{x}\big\}$, $\mathcal{A}_q := \big\{m\in \widetilde{\mathcal{A}}: q\text{ divides }m\big\}$, and $\mathcal{A}_{\widetilde{q}} := \big\{m\in \widetilde{\mathcal{A}}: \widetilde{q}\text{ divides }m\big\}$. Note that $\widetilde{\mathcal{A}}\backslash\mathcal{A}_{\widetilde{q}}$ is a revlex-segment set in $\{W_d(\phi)\}$, while $\mathcal{A}_q \cap \mathcal{A}_{\widetilde{q}} = \emptyset$, since $qx_{s,2}^{\dag} = \widetilde{q}x_{s,2}$ is not contained in $\widetilde{\mathcal{A}}$. Also, $q_{X_t} = \widetilde{q}_{X_t}$ for every $t\in \Gamma$, thus for each $x\in \bigcup_{t\in \Gamma} X_t$, we have $qx \in \{W_d(\phi)\}$ if and only if $\widetilde{q}x \in \{W_d(\phi)\}$. The definition of $\widehat{p}$ implies $qx_{s,1}^{\dag}$, $qx_{\overline{s-1},1}$, $\widetilde{q}x_{s,1}^{\dag}$, $\widetilde{q}x_{\overline{s-1},1}$ are monomials in $W_d(\phi)$, while the definition of $\widehat{m}_s$ implies $qx_{s,1} \not\in \{W_d(\phi)\}$ and $\widetilde{q}x_{s,1} \in \{W_d(\phi)\}$. Consequently, $|\mathcal{A}_q| = |\mathcal{A}_{\widetilde{q}}| - 1 < |\mathcal{A}_{\widetilde{q}}|$.

Now, define $\mathcal{A} := \widetilde{\mathcal{A}}\backslash \mathcal{A}_q$, and let $x_0$ be the (unique) variable in $X_{\boldsymbol{\lambda}}$ such that $\widetilde{q}x_0$ is the largest monomial in $\mathcal{A}_{\widetilde{q}}$. The set $\mathcal{I} := (\widetilde{\mathcal{A}}\backslash\mathcal{A}_{\widetilde{q}}) \cup \{\widetilde{q}x_0\}$ is revlex-segment in $\{W_d(\phi)\}$, and $|\mathcal{I}| = |\mathcal{A}|$ implies $\Spank(\mathcal{I}) = \Revlex_{W_d(\phi)}(\Spank(\mathcal{A}))$. Furthermore, for any $x, x^{\prime} \in X_{\boldsymbol{\lambda}}$ such that $\widetilde{q}x \in \widetilde{\mathcal{A}}$, $x^{\prime} \neq x$, and $x^{\prime}$ divides $\widetilde{q}$, we claim that $\frac{\widetilde{q}x}{x^{\prime}} \in \partial(\widetilde{\mathcal{A}}\backslash \mathcal{A}_{\widetilde{q}})$. Indeed, let $t^{\prime}$ be the unique integer in $[n]$ satisfying $x^{\prime} \in X_{t^{\prime}}$, and note that the construction of $\widetilde{q}$ yields $x^{\prime} \neq x_{t^{\prime},1}$, hence $\frac{\widetilde{q}x}{x^{\prime}}\cdot x_{t^{\prime},1} \in \widetilde{\mathcal{A}}\backslash\mathcal{A}_{\widetilde{q}}$, which means $\frac{\widetilde{q}x}{x^{\prime}} \in \partial(\widetilde{\mathcal{A}}\backslash \mathcal{A}_{\widetilde{q}})$ as claimed. This implies $\partial(\mathcal{I}) = \partial(\widetilde{\mathcal{A}})$, so since $\partial(\mathcal{A}) \subseteq \partial(\widetilde{\mathcal{A}})\backslash\{q\}$, we conclude that $|\partial(\Spank(\mathcal{A}))| < |\partial( \widetilde{\mathcal{A}})| = |\partial(\mathcal{I})| = |\partial(\Revlex_{W_d(\phi)}(\Spank(\mathcal{A})))|$.
\end{proof}

\begin{remark}
In the proof of Proposition \ref{prop:counter-eg-a>d+1}, we used the condition $\boldsymbol{\lambda} \geq 2\cdot {\bf 1}_n + \boldsymbol{\delta}_n$ only in the case when $s=1$, i.e. when every $t\in [n-1]$ satisfying $a_{t+1} \leq d-1$ also satisfies $a_t = 1$. In all other cases, i.e. when $a_{t+1}\leq d-1$ and $a_t \geq 2$ for some $t\in [n-1]$, we only need the weaker condition that $\boldsymbol{\lambda} \geq 2\cdot {\bf 1}_n$.
\end{remark}

\begin{corollary}\label{cor:truncatedNotMacLex}
Let $n>1$, ${\bf a} \in \overline{\mathbb{P}}^n$, $\boldsymbol{\lambda} \in \mathbb{P}^n$ such that $\boldsymbol{\lambda} \geq 2\cdot {\bf 1}_n + \boldsymbol{\delta}_n$ and ${\bf a} \neq {\bf 1}_n$. If $\phi: X_{\boldsymbol{\lambda}} \to \overline{\mathbb{P}}$ is a map such that $\sum_{x\in X_t\backslash \{x_{t,1}\}} \phi(x) \geq a_t$ for every $t\in [n]$, then $W(\phi)$ is not Macaulay-Lex.
\end{corollary}

\begin{proof}
Setting $d = \min\{a_i+1: i \in [n]\}$, we clearly have ${\bf a} \not\geq d\cdot {\bf 1}_n$, while the condition ${\bf a} \neq {\bf 1}_n$ implies $|{\bf a}| > d$. The assertion then follows from Proposition \ref{prop:counter-eg-a>d+1} and Theorem \ref{thm:charMacLexRings}.
\end{proof}

Let $X$ be a non-empty linearly ordered set. For each $d\in \mathbb{N}$, recall that the revlex order $\leq_{r\ell}$ on $\mathcal{M}_X^d$ (induced by the linear order on $X$) was already defined in Section \ref{sec:MacLexRings} for finite $X$, and when $X$ is infinite, we extend $\leq_{r\ell}$ in the obvious way. Given a non-empty subcollection $\mathcal{M} \subseteq \mathcal{M}_X$, let $\mathcal{M}^d := \mathcal{M} \cap \mathcal{M}_X^d$, and define a {\it revlex-segment set in $\mathcal{M}^d$} to be a finite subset $\mathcal{A} \subseteq \mathcal{M}^d$ such that if $m\in \mathcal{A}$, $m^{\prime} \in \mathcal{M}^d$ and $m^{\prime} >_{r\ell} m$, then $m^{\prime} \in \mathcal{A}$. If $M$ is a multicomplex on $X$, then $\Spank(\mathcal{M}_X\backslash M)$ forms a (graded) monomial ideal of the standard graded polynomial ring $\fieldk[X]$, which we denote by $I^M = \bigoplus_{d\in \mathbb{N}} I_d^M$. The map $M \mapsto I^M$ gives a bijection between multicomplexes on $X$ and monomial ideals of $\fieldk[X]$, and the corresponding inverse map is $I \mapsto \{\fieldk[X]/I\}$. We say $M$ is a {\it compressed multicomplex in $\mathcal{M}$} if $M$ is finite, and $M^d$ is a revlex-segment set in $\mathcal{M}^d$ for every $d\in \mathbb{N}$. When $\mathcal{M}$ is not specified, we assume $\mathcal{M} = \mathcal{M}_X$.

\begin{lemma}\label{lemma:MacLexMulticomplexCorresp}
Let $M$ be a multicomplex on $X$. The following are equivalent:
\begin{enum*}
\item For every multicomplex $M^{\prime} \subseteq M$, there exists a compressed multicomplex $M^{\prime\prime}$ in $M$ with the same $f$-vector as $M^{\prime}$.
\item $I^M$ is a Macaulay-Lex ideal of the polynomial ring $\fieldk[X]$.
\end{enum*}
\end{lemma}

\begin{proof}
For every graded ideal $I$ of $R := \fieldk[X]/I^M$, let $\inid(I)$ be its initial ideal with respect to some monomial order (see \cite[Chap. 15]{Eisenbud:CommutativeAlgebraBook}). Hilbert functions are additive on exact sequences, so by \cite[Thm. 15.26]{Eisenbud:CommutativeAlgebraBook}, $R$ is Macaulay-Lex if and only if for every graded ideal $I$ of $R$, there exists a lex ideal $L$ of $R$ such that $R/\inid(I)$ and $R/L$ have the same Hilbert function, or equivalently, $\{R\}\backslash \{\inid(I)\}$ and $\{R\}\backslash \{L\}$ are multicomplexes with the same $f$-vectors. By definition, $L$ is a lex ideal of $R$ if and only if $\{R\}\backslash \{L\}$ is a compressed multicomplex in $\{R\}$. Now, for every multicomplex $M^{\prime} \subseteq M$, the image of $I^{M^{\prime}}$ under the natural quotient map $\fieldk[X] \to R$ is a monomial ideal of $R$, thus the assertion follows.
\end{proof}

For the rest of this section, let ${\bf a}\in \mathbb{P}^n$, let $\pi = (X_1, \dots, X_n)$ be an ordered partition of a non-empty set $X$, and let $\phi: X\to \overline{\mathbb{P}}$ be a map. For each $t\in [n]$, label the elements in $X_t$ by $X_t = \{x_{t,1}, x_{t,2}, \dots \}$, and fix a linear order on $X$ by $x_{i,j} > x_{i^{\prime}, j^{\prime}}$ if $j > j^{\prime}$; or $j = j^{\prime}$ and $i < i^{\prime}$ (cf. Definition \ref{defn-coloredQuot}). Define $\mathcal{M}_{{\bf a}, \pi} := \big\{m\in \mathcal{M}_X: \deg(m_{X_t}) \leq a_t\text{ for all }t\in [n]\big\}$, and let $\mathcal{M}_{{\bf a}, \pi}(\phi) := \mathcal{M}_{{\bf a}, \pi} \cap \mathcal{M}_X(\phi)$. Note that $(M, \pi)$ is an ${\bf a}$-colored multicomplex in $\mathcal{M}_X(\phi)$ if and only if $M$ is a multicomplex in $\mathcal{M}_{{\bf a}, \pi}(\phi)$. We say $(M, \pi)$ is {\it revlex} (reverse-lexicographic) if $M$ is a compressed multicomplex in $\mathcal{M}_{{\bf a}, \pi}(\phi)$ with respect to this given linear order on $X$. In particular, recall that $x_{i,j} > x_{i^{\prime}, j^{\prime}} \Leftrightarrow x_{i,j} <_{r\ell} x_{i^{\prime}, j^{\prime}}$.

\begin{proposition}\label{prop:revlexmultiequiv}
The following are equivalent:
\begin{enum*}
\item For every ${\bf a}$-colored multicomplex $(M,\pi)$ in $\mathcal{M}_X(\phi)$, there exists a revlex ${\bf a}$-colored multicomplex $(M^{\prime}, \pi)$ in $\mathcal{M}_X(\phi)$ such that $M$ and $M^{\prime}$ have the same $f$-vector.
\item For every $\boldsymbol{\lambda} \in \mathbb{P}^n$ satisfying $\boldsymbol{\lambda} \leq (|X_1|, \dots, |X_n|)$, the truncation of the colored quotient ring $\fieldk[X_{\boldsymbol{\lambda}}]/Q_{{\bf a}}$ induced by $\phi$ is Macaulay-Lex.
\end{enum*}
\end{proposition}

\begin{proof}
Let $|\pi| := (|X_1|, \dots, |X_n|) \in \overline{\mathbb{P}}^n$, and let $X_t^{[r]}$ be the subset $\{x_{i,t}: i\in [r]\} \subseteq X_t$ for every $r\in \mathbb{P}$, $t\in [n]$. For each ${\boldsymbol{\lambda}} = (\lambda_1, \dots, \lambda_n) \in \mathbb{P}^n$ satisfying $\boldsymbol{\lambda} \leq |\pi|$, consider $X_{\boldsymbol{\lambda}}$ as a subposet of $X$, let $\pi_{\boldsymbol{\lambda}}$ be the ordered partition $(X_1^{[\lambda_1]}, \dots, X_n^{[\lambda_n]})$ of $X_{\boldsymbol{\lambda}}$, and let $\phi_{\boldsymbol{\lambda}}: X_{\boldsymbol{\lambda}} \to \overline{\mathbb{P}}$ be the restriction of $\phi$ to $X_{\boldsymbol{\lambda}}$ as its domain.

Suppose $M$ is a multicomplex in $\mathcal{M}_{{\bf a}, \pi}(\phi)$. Choose any $\boldsymbol{\lambda} \in \mathbb{P}^n$ such that $\boldsymbol{\lambda} \leq |\pi|$ and $M \subseteq \mathcal{M}_{X_{\boldsymbol{\lambda}}}$, which is possible since ${\bf a}\in \mathbb{P}^n$ implies $M$ is finite. Every compressed multicomplex $M^{\prime}$ in $\mathcal{M}_{{\bf a}, \pi}(\phi)$ with the same $f$-vector as $M$ is necessarily contained in $\mathcal{M}_{{\bf a}, \pi_{\boldsymbol{\lambda}}}(\phi_{\boldsymbol{\lambda}})$, hence $(M^{\prime}, \pi)$ is a revlex ${\bf a}$-colored multicomplex in $\mathcal{M}_{X_{\boldsymbol{\lambda}}}(\phi_{\boldsymbol{\lambda}}) \subseteq \mathcal{M}_X(\phi)$ for each such compressed multicomplex $M^{\prime}$. Since $\mathcal{M}_{X_{\boldsymbol{\lambda}}}(\phi_{\boldsymbol{\lambda}})$ is a multicomplex on $X_{\boldsymbol{\lambda}}$, Lemma \ref{lemma:MacLexMulticomplexCorresp} says that for every $\boldsymbol{\lambda} \in \mathbb{P}^n$ satisfying $\boldsymbol{\lambda} \leq |\pi|$, the following are equivalent:
\begin{itemize}
\item For every ${\bf a}$-colored multicomplex in $\mathcal{M}_X(\phi) \cap \mathcal{M}_{X_{\boldsymbol{\lambda}}}$, there exists a revlex ${\bf a}$-colored multicomplex $(M^{\prime}, \pi)$ in $\mathcal{M}_X(\phi)$ such that $M$ and $M^{\prime}$ have the same $f$-vector.
\item The ideal $I^{\mathcal{M}_{X_{\boldsymbol{\lambda}}}(\phi_{\boldsymbol{\lambda}})} = \big(\sum_{i\in [n]} \big\langle X_i^{[\lambda_i]}\big\rangle^{a_i+1}\big) + \big\langle\big\{ x^{\phi(x)+1}: x\in X_{\boldsymbol{\lambda}}\big\} \big\rangle$ is a Macaulay-Lex ideal of $\fieldk[X_{\boldsymbol{\lambda}}]$.
\end{itemize}
The assertion then follows by considering all possible $\boldsymbol{\lambda} \in \mathbb{P}^n$ satisfying $\boldsymbol{\lambda} \leq |\pi|$.
\end{proof}

\begin{theorem}\label{thm:coloredKK-multicomplexVer}
For every ${\bf 1}_n$-colored multicomplex $(M, \pi)$, there exists a revlex ${\bf 1}_n$-colored multicomplex $(M^{\prime}, \pi)$ such that $M$ and $M^{\prime}$ have the same $f$-vector.
\end{theorem}

\begin{proof}
This follows from Proposition \ref{prop:revlexmultiequiv} and the case ${\bf a} = {\bf 1}_n$ of Theorem \ref{main-thm-MacLex}.
\end{proof}

\begin{lemma}\label{lemma:truncated-n=1-case}
Let $n\in \mathbb{P}$ and $a, \alpha_1, \dots, \alpha_n \in \overline{\mathbb{P}}$. Then the ideal $I := \big\langle x_1, \dots, x_n\big\rangle^{a+1} + \big\langle x_1^{\alpha_1+1}, \dots, x_n^{\alpha_n+1} \big\rangle$ in $S$ is Macaulay-Lex if and only if $\min\{a, \alpha_1\} \leq \dots \leq \min\{a, \alpha_n\}$. [Note: By definition, $x_i^{\infty} = 0$.]
\end{lemma}

\begin{proof}
Let $R = \bigoplus_{d\in \mathbb{N}} R_d := S/I$. The case $a = \infty$ follows from Remark \ref{remark:caseCompositionAllOnes}, so assume $a<\infty$. Since every $R_d$-monomial space satisfying $d>a$ is empty, Theorem \ref{thm:charMacLexRings} says that $R$ is Macaulay-Lex if and only if $\partial(\Revlex_{R_d}(A)) \subseteq \Revlex_{R_{d-1}}(\partial(A))$ for all $d\in [a]$ and every $R_d$-monomial space $A$. Fix some $d\in [a]$, and choose an arbitrary $R_d$-monomial space $A$. Next, let $\beta_i = \min\{a, \alpha_i\}$ for each $i\in [n]$, let $T = \bigoplus_{d\in \mathbb{N}} T_d := S/J$, where $J := \langle x_1^{\beta_1+1}, \dots, x_n^{\beta_n+1}\rangle \subseteq S$, and notice that Remark \ref{remark:caseCompositionAllOnes} says $T$ is Macaulay-Lex if and only if $\beta_1 \leq \dots \leq \beta_n$. Since $d\leq a$ implies $\{R_d\} = \{T_d\}$, it then follows from Theorem \ref{thm:charMacLexRings} that $\partial(\Revlex_{R_d}(A)) = \partial(\Revlex_{T_d}(A) \subseteq \Revlex_{T_{d-1}}(\partial(A)) = \Revlex_{R_{d-1}}(\partial(A))$ if and only if $\beta_1 \leq \dots \leq \beta_n$, and the assertion follows.
\end{proof}

\begin{theorem}\label{thm:revlexdoesntcharacterize}
If $\pi$ and $\phi$ satisfy $(|X_1|, \dots, |X_n|) \geq 2\cdot {\bf 1}_n + \boldsymbol{\delta}_n$, and $\sum_{x\in X_t\backslash\{x_{t,1}\}} \phi(x) \geq a_t$ for all $t\in [n]$, then the following are equivalent:
\begin{enum*}
\item For every ${\bf a}$-colored multicomplex $(M, \pi)$ in $\mathcal{M}_X(\phi)$, there exists a revlex ${\bf a}$-colored multicomplex $(M^{\prime}, \pi)$ in $\mathcal{M}_X(\phi)$ such that $M$ and $M^{\prime}$ have the same $f$-vector.\label{cond:revlexcharwhen}
\item Either $n=1$ and $x > x^{\prime}$ for all $x, x^{\prime} \in X$ satisfying $\min\{a_1, \phi(x)\} < \min\{a_1, \phi(x^{\prime})\}$ (if any); or ${\bf a} = {\bf 1}_n$.\label{cond:revlexcharexplicit}
\end{enum*}
\end{theorem}

Before proving Theorem \ref{thm:revlexdoesntcharacterize}, observe that Theorem \ref{thm:revlexdoesntcharacterize} implies Theorem \ref{main-thm-f-vectors-colored}. Indeed, for each ${\bf a}$-colored multicomplex $(M, \pi)$ in $\mathcal{M}_X(\phi)$, since we do not require every $x\in X$ to be in $M$, we may assume without loss of generality that $(|X_1|, \dots, |X_n|) = \boldsymbol{\infty}_n$ by adding new variables if necessary. Consequently, the map $\phi$ trivially satisfies $\sum_{x\in X_t\backslash\{x_{t,1}\}} \phi(x) = \infty \geq a_t$ for every $t\in [n]$, thus by considering $\phi = {\bf 1}$ and the constant map $\phi:X \to \overline{\mathbb{P}}$ defined by $x\mapsto \infty$ for all $x\in X$, Theorem \ref{main-thm-f-vectors-colored} follows from Theorem \ref{thm:revlexdoesntcharacterize}.

\vspace{0.65em}
\begin{paragraph}{\it Proof of Theorem \ref{thm:revlexdoesntcharacterize}.}
If $n>1$ and \ref{cond:revlexcharwhen} holds, then ${\bf a} = {\bf 1}_n$ by Corollary \ref{cor:truncatedNotMacLex} and Proposition \ref{prop:revlexmultiequiv}. Conversely, if $n>1$ and ${\bf a} = {\bf 1}_n$, then \ref{cond:revlexcharwhen} holds by Theorem \ref{thm:coloredKK-multicomplexVer}. As for the case $n=1$, i.e. ${\bf a} = (a_1) \in \mathbb{P}$ and $\pi = (X)$, Proposition \ref{prop:revlexmultiequiv} and Lemma \ref{lemma:truncated-n=1-case} together imply \ref{cond:revlexcharwhen} holds if and only if $\min\{a_1, \phi(x_{1,1})\} \geq \dots \geq \min\{a_1, \phi(x_{1,\lambda})\}$ for every $\lambda\in \mathbb{P}$ satisfying $\lambda \leq |X|$, or equivalently, $x > x^{\prime}$ for all $x, x^{\prime} \in X$ satisfying $\min\{a_1, \phi(x)\} < \min\{a_1, \phi(x^{\prime})\}$ (if any).
\end{paragraph}

\section{Further remarks}\label{sec:FurtherRemarks}
\subsection{Colored Kruskal-Katona theorem}\label{subsec:ColoredKKThm}
The Frankl-F\"{u}redi-Kalai theorem~\cite{FFK1988} is sometimes known as the colored Kruskal-Katona theorem, since it is an extension of the Kruskal-Katona theorem to colored $k$-uniform hypergraphs. Although the Frankl-F\"{u}redi-Kalai theorem was originally formulated in terms of the minimum shadow sizes of $k$-uniform $n$-colored hypergraphs (see \cite{London1994}), it is equivalent to Theorem \ref{thm:coloredKK-multicomplexVer} in the case when $\boldsymbol{\lambda} = (\lambda_1, \dots, \lambda_n) := (|X_1|, \dots, |X_n|) \in \overline{\mathbb{P}}^n$ satisfies $\lambda_1 \leq \dots \leq \lambda_n \leq \lambda_1 + 1$. This includes the special case $\boldsymbol{\lambda} = \boldsymbol{\infty}_n$.

Note that Claim 4.2(ii) in \cite{FFK1988} is not true, hence the original proof of the Frankl-F\"{u}redi-Kalai theorem in \cite{FFK1988} is incorrect as stated. However, London~\cite{London1994} gave a different (and correct) proof, Engel~\cite[Chap. 8]{book:EngelSpernerTheory} gave a proof using properties of Macaulay posets, while Mermin and Murai~\cite{MerminMurai2010} gave a proof in the language of monomial ideals, so the statement of the theorem is still true. In particular, London proved Theorem \ref{thm:coloredKK-multicomplexVer} in the case when $\lambda_1 \leq \dots \leq \lambda_n \leq \lambda_1 + 1$, while both Engel and Mermin-Murai proved Theorem \ref{thm:coloredKK-multicomplexVer} in the case when $\lambda_n \leq \dots \leq \lambda_1$. In fact, Theorem \ref{thm:coloredKK-multicomplexVer} holds for arbitrary $\boldsymbol{\lambda} \in \overline{\mathbb{P}}^n$.

On a related note, it was pointed out in \cite{book:BilleraBjornerHandbookDiscCompGeom} that the uniqueness claim in \cite[Lem. 1.1]{FFK1988} is incorrect, which makes the $\partial_k^{(r)}(\cdot)$ operator introduced in \cite{FFK1988} not well-defined. The numerical version of the Frankl-F\"{u}redi-Kalai theorem stated in \cite{book:BilleraBjornerHandbookDiscCompGeom} includes a fix suggested by J. Eckhoff.

\subsection{Fine $f$-vectors of generalized colored complexes}\label{subsec:FineFVectors}
Let ${\bf a} \in \mathbb{P}^n$, and let $(\Delta, \pi)$ be an ${\bf a}$-colored complex. For each ${\bf b} = (b_1, \dots, b_n) \in \mathbb{N}^n$ satisfying ${\bf b} \leq {\bf a}$, let $f_{{\bf b}}$ be the number of faces $F$ of $\Delta$ such that $|F \cap V_i| = b_i$. The array of integers $\{f_{{\bf b}}\}_{{\bf b} \leq {\bf a}}$ is called the {\it fine $f$-vector} of $(\Delta, \pi)$, and it is a refinement of the $f$-vector $(f_0, \dots, f_{\dim \Delta})$ of $\Delta$ in the sense that $f_i = \sum_{|{\bf b}| = i+1} f_{{\bf b}}$ for all $0 \leq i \leq \dim \Delta$. As pointed out by Bj\"{o}rner-Frankl-Stanley~\cite{BjornerFranklStanley1987:CohenMacaulayComplexes}, part of the difficulty in finding a numerical characterization of the fine $f$-vectors of arbitrary ${\bf a}$-colored complexes lies in the non-uniqueness of color-compressed ${\bf a}$-colored complexes with a given fine $f$-vector. Furthermore, even for (the less refined) $f$-vectors, Theorem \ref{main-thm-f-vectors-colored} tells us that the numerical characterization of the $f$-vectors of ${\bf 1}_n$-colored complexes, proven by Frankl-F\"{u}redi-Kalai~\cite{FFK1988}, does not extend to arbitrary ${\bf a} \in \mathbb{P}^n$.

An {\it ${\bf a}$-balanced} complex is an ${\bf a}$-colored complex $(\Delta, \pi)$ satisfying $|{\bf a}| = \dim \Delta + 1$. In particular, we say $(\Delta, \pi)$ is {\it completely balanced} if ${\bf a} = {\bf 1}_{(\dim \Delta + 1)}$, and important examples include Coxeter complexes and order complexes of posets. Recently, the author~\cite{Chong2013:GenMacReps} introduced the notion of Macaulay decomposability for simplicial complexes and used it to obtain a numerical characterization of the fine $f$-vectors of ${\bf a}$-colored complexes for all ${\bf a} \in \mathbb{P}^n$. Via a unified approach, he also gave a numerical characterization of the fine $f$-vectors of completely balanced Cohen-Macaulay complexes in \cite{Chong2013:GenMacReps}.

\section*{Acknowledgements}
The author thanks the anonymous referees for valuable comments that greatly improved the presentation of the paper. The author also thanks the anonymous referee who suggested the current proof of Theorem \ref{thm:color-comp}, which is significantly shorter than the original proof.

\bibliographystyle{plain}
\bibliography{References}

\end{document}